\documentclass[11pt,reqno,a4paper]{amsart}
\usepackage[norelsize]{algorithm2e}
\usepackage{ma,bm}

\author{Jun Kitagawa}
\address{Department of Mathematics, Michigan State University}
\email{kitagawa@math.msu.edu}

\author{Quentin M\'erigot}
\address{Laboratoire de Mathématiques d'Orsay, Univ. Paris-Sud, CNRS,
  Université Paris-Saclay, 91405 Orsay, France}
\email{quentin.merigot@math.u-psud.fr}

\author{Boris Thibert}
\address{Laboratoire Jean Kuntzmann, Université Grenoble-Alpes}
\email{boris.thibert@univ-grenoble-alpes.fr}

\title[A Newton algorithm for semi-discrete optimal
  transport]{Convergence of a Newton algorithm for semi-discrete
  optimal transport}

\begin{document}
\begin{abstract}
  A popular way to solve optimal transport problems numerically
    is to assume that the source probability measure is absolutely
    continuous while the target measure is finitely supported. We
    introduce a damped Newton's algorithm in this setting, which is
  experimentally efficient, and we establish its global linear
  convergence for cost functions satisfying an assumption that appears
  in the regularity theory for optimal transport.
\end{abstract}

\maketitle

\tableofcontents
\section{Introduction}
Some problems in geometric optics or convex geometry can be recast as
optimal transport problems between probability measures: this includes
the far-field reflector antenna problem, Alexandrov's Gaussian
curvature prescription problem, etc. A popular way to solve these
problems numerically is to assume that the source probability measure
is absolutely continuous while the target measure is finitely
supported. We refer to this setting as semi-discrete optimal
transport. Among the several algorithms proposed to solve
semi-discrete optimal transport problems, one currently needs to
choose between algorithms that are slow but come with a convergence
speed analysis
\cite{oliker1989numerical,caffarelli1999problem,kitagawa2014iterative}
or algorithms that are much faster in practice but which come with no
convergence guarantees
\cite{aurenhammer1998minkowski,merigot2011multiscale,de2012blue,levy2014numerical,de2014intersection}.
Algorithms of the first kind rely on coordinate-wise increments and
the number of iterations required to reach the solution up to an error
of $\eps$ is of order $N^3/\eps$, where $N$ is the number of Dirac
masses in the target measure. On the other hand, algorithms of the
second kind typically rely on the formulation of the semi-discrete
optimal transport problem as an unconstrained convex optimization
problem which is solved using a Newton or quasi-Newton method.

The purpose of this article is to bridge this gap between theory and
practice by introducing a damped Newton's algorithm which is
experimentally efficient and by proving the global convergence of this
algorithm with optimal rates. The main assumptions is that the cost
function satisfies a condition that appears in the regularity theory
for optimal transport (the Ma-Trudinger-Wang condition) and that 
the support of the source density is connected in a quantitative way
(it must satisfy a weighted Poincaré-Wirtinger inequality). In
\S\ref{intro:convergence}, we compare this algorithm and the
convergence theorem to previous computational approaches to optimal
transport.

\subsection{Semi-discrete optimal transport}
The source space is an open domain $\Omega$ of a $d$-dimensional
Riemannian manifold, which we endow with the measure $\Haus_g^d$ induced
by the Riemannian metric $g$ on the manifold.  The target space is
an (abstract) finite set $Y$. We are given a cost function $c$ on the
product space $\Omega \times Y$, or equivalently a collection
$(c(\cdot,y))_{y\in Y}$ of functions on $\Omega$. We assume that the
functions $c(\cdot,y)$ are of class $\Class^{1,1}$ on $\Omega$:
\begin{equation}
\forall y \in Y,~~ c(\cdot,y) \in \Class^{1,1}(\Omega) \tag{Reg} \label{eq:Reg}.
\end{equation}
Here $\Class^{n,\alpha}(\Omega)$ denotes the class of functions which are $n$-times differentiable and whose $n$-th derivatives are $\alpha$-H\"{o}lder continuous. In particular, $\Class^{0,\alpha}$ is the space of  $\alpha$-H\"{o}lder continuous functions.
We consider a compact subset $X$ of $\Omega$ and $\rho$ a probability
density on $X$, i.e. such that $\rho \dd \Haus^d$ is a probability
measure. By an abuse of notation, we will often conflate the density $\rho$ with the measure $\rho \dd \Haus^d$ itself.  
 Note that the support of $\rho$ is contained in $X$, but we
do not assume that it is actually equal to $X$. The push-forward of
$\rho$ by a measurable map $T:X\to Y$ is the finitely supported
measure $T_{\#}\rho = \sum_{y\in Y} \rho(T^{-1}(y)) \delta_y$. The map
$T$ is called a \emph{transport map} between $\rho$ and a probability
measure $\mu$ on $Y$ if $T_\#\rho = \mu$. The \emph{semi-discrete
  optimal transport problem} consists in minimizing the transport cost
over all transport maps between $\rho$ and $\mu$, that is
\begin{equation}
  \min \left\{ \int_{X} c(x,T(x)) \rho(x) \dd \Haus_g^{d}(x) \mid T:X\to Y \hbox{ s.t. } T_{\#} \rho = \mu \right\}. \tag{M}
\label{eq:M}
\end{equation}
This problem is an instance of Monge's optimal transport problem,
where the target measure is finitely supported.
Kantorovich proposed a relaxed version of the problem \eqref{eq:M} as
an infinite dimensional linear programming problem over the space of
probability measures with marginals $\rho$ and $\mu$.

\subsection{Laguerre tessellation and economic interpretation}
\begin{figure}
    \begin{center}
      \includegraphics[width=\textwidth]{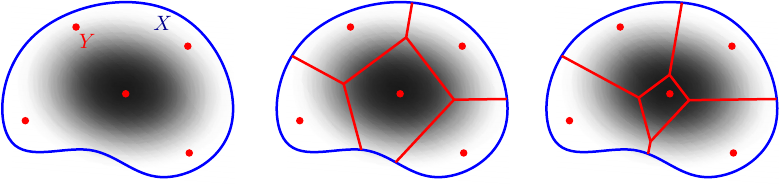}
      \caption{(Left) The domain $X$ (with boundary in blue) is
        endowed with a probability density pictured in grayscale
        representing the density of population in a city. The set $Y$
        (in red) represents the location of bakeries. Here,
        $X,Y\subseteq \Rsp^2$ and $c(x,y) = |x-y|^2$ (Middle) The
        Voronoi tessellation induced by the bakeries (Right) The
        Laguerre tessellation: the price of bread the bakery near the
        center of $X$ is higher than at the other bakeries,
        effectively shrinking its Laguerre cell.\label{fig:laguerre}}
      \end{center}
\end{figure}
In the semi-discrete setting, the dual of Kantorovich's relaxation can
be conveniently phrased using the notion of Laguerre tessellation. We
start with an economic metaphor. Assume that the probability density
$\rho$ describes the population distribution over a large city $X$, and that the
finite set $Y$ describes the location of bakeries in the
city. Customers living at a location $x$ in $X$ try to minimize the
walking cost $c(x,y)$, resulting in a decomposition of the space called a
Voronoi tessellation. The number of customers received by a bakery $y\in Y$
is equal to the integral of $\rho$ over its Voronoi cell, namely
$$\Vor_y := \{ x\in \Omega\mid\forall z \in Y, c(x,y) \leq
c(x,z)  \}. $$ If the price of bread is given by a function
$\psi: Y\to\Rsp$, customers living at location $x$ in $X$ make a
compromise between walking cost and price by minimizing the sum
$c(x,y) + \psi(y)$. This leads to the notion of Laguerre tessellation,
whose cells are given by
\begin{equation}
  \Lag_y(\psi) := \{ x\in \Omega \mid \forall z \in Y, c(x,y) + \psi(y) \leq 
  c(x,z) + \psi(z) \}.
\end{equation}
When the sets $X$ and $Y$ are contained in $\Rsp^d$ and the cost is the squared
Euclidean distance, the computation of the Laguerre tessellation is a
classical problem of computational geometry, for which there exists
very efficient softwares, such as CGAL~\cite{cgal} or Geogram~\cite{geogram}. For other cost functions, one has to adapt the algorithms, as was done for the reflector cost on the sphere in~\cite{de2014intersection}. The shape of the Voronoi and Laguerre tessellations
is depicted in Figure~\ref{fig:laguerre}.

We want the Laguerre cells to form a partition of $\Omega$ up to a
negligible set. By the implicit function theorem, this will be the
case if the  following \emph{twist condition} holds,
\begin{equation}
\forall x\in X,~~ y\in Y\mapsto D_x c(x,y)\in T^*_x\Omega \hbox{ is injective,} 
\tag{Twist} \label{eq:Tw}
\end{equation}
where $D_x$ denotes differentiation with respect to the first variable. The twist condition implies that for any prices
$\psi$ on $Y$, the transport map induced by the Laguerre tessellation
\begin{equation}
 T_\psi(x) := \arg\min_{y \in Y} (c(x,y) + \psi(y)), \label{eq:Tpsi}
\end{equation}
is uniquely defined almost everywhere. It is easy to see (see
Proposition~\ref{prop:OT}), that for any function $\psi$ on $Y$, the
map $T_\psi$ is an optimal transport map between $\rho$ and the
pushforward measure $T_{\psi \#} \rho = \sum_{y\in Y}
\rho(\Lag_y(\psi)) \delta_y.$

\subsection{Kantorovich's functional}The map $T_\psi$ is an optimal transport map
between $\rho$ and  $T_{\psi \#} \rho$. Conversely, Theorem~\ref{th:Aurenhammer} below ensures that any
semi-discrete optimal transport problem admits such a solution. In
other words, for any probability  density $\rho$ on $X$ and any probability measures $\mu$ on $Y$ there exists a
function (price) $\psi$ on $Y$ such that $T_{\psi \#}\rho =\nu$.  The
proof of this theorem is an easy generalization of the proof given in
\cite{aurenhammer1998minkowski} for the quadratic cost, but it is
nonetheless included in Section~\ref{sec:Aurenhammer} for the sake of
completeness.

Here and after, we denote $(\one_y)_{y\in Y}$ the canonical basis of
$\Rsp^Y$, and $\nr{\cdot}$ the Euclidean norm induced by this basis, while $\nr{\cdot}_{g}$ will denote the norm induced by the Riemannian metric $g$ on either $T_x\Omega$ or $T^*_x\Omega$ (which will be clear from context).  We
will, slightly abusively, consider the space of probability measures 
$\Prob(Y)$ as a subset of $\Rsp^Y$.

\begin{theorem}
  \label{th:Aurenhammer} Assume \eqref{eq:Reg} and \eqref{eq:Tw}, let $\rho$ be a bounded probability
  density on $X$ and $\nu = \sum_{y\in Y} \nu_y\one_y$ in
  $\Prob(Y)$. Then, the functional $\Phi$
\begin{align}
  \Phi(\psi) &:= \int_{X} (\min_{y\in Y} c(x,y) + \psi(y))\rho(x)
 \dd\Haus_g^{d}(x) - \sum_{y\in Y}\psi(y)\nu_y \notag \\
&= \sum_{y\in Y} \int_{\Lag_y(\psi)} (c(x,y) + \psi(y)) \rho(x) \dd\Haus_g^d(x) - \sum_{y\in Y}\psi(y)\nu_y
\label{eq:Phi}
\end{align}
is concave, $\Class^1$-smooth, and its gradient is
\begin{equation}
  \nabla \Phi(\psi) = \sum_{y\in Y} (\rho(\Lag_y(\psi)) - \nu_y) \one_y.
\label{eq:GradPhi}
\end{equation}
\end{theorem}

\begin{corollary} The following statements are equivalent:
\begin{itemize}
\item[(i)] $\psi: Y\to \Rsp$ is a global maximizer of $\Phi$ ;
\item[(ii)] $T_\psi$ is an optimal transport map between $\rho$ and
  $\nu$ ;
\item[(iii)] $T_{\psi \#} \rho = \nu$, or equivalently, 
\begin{equation}
\forall y\in Y,~\rho(\Lag_y(\psi)) = \nu_y \tag{MA} \label{eq:MA}
\end{equation}
\end{itemize}
\end{corollary}

We call \emph{Kantorovich's functional} the function $\Phi$ introduced
in \eqref{eq:Phi}. Note that both this functional and its gradient are
invariant by addition of a constant.  The non-linear equation
\eqref{eq:MA} can be considered as a discrete version of the
generalized Monge-Ampère equation that characterizes the solutions to
optimal transport problems (see for instance Chapter~12 in
\cite{villani2009optimal}).

\subsection{Damped Newton algorithm}\label{subsection:algo}
We consider a simple damped Newton's algorithm to solve semi-discrete
optimal transport problem.  This algorithm is very close to the one
used by Mirebeau in \cite{mirebeau2015discretization}. To phrase this algorithm in a more general way, we
introduce a notation for the measure of Laguerre cells: for $\psi \in
\Rsp^Y$ we set
 \begin{equation}
   \Gglobal(\psi) := \sum_{y \in Y} G_y(\psi) \one_y \hbox{ where } G_y(\psi) = \rho(\Lag_y(\psi)),
 \end{equation}
so that $\nabla \Phi(\psi) = \Gglobal(\psi) - \nu$.  
In the algorithm (Algorithm \ref{algo:newton}), we denote by $A^+$ the
\emph{pseudo-inverse} of the matrix $A$.

\begin{algorithm}
\begin{description}
  \item[Input] A tolerance $\eta > 0$ and an initial $\psi_0\in
    \Rsp^Y$ such that
\begin{equation}\label{eq:nonzerocells}
  \eps_0 :=  \frac{1}{2} 
  \min\left[\min_{y\in Y} G_y(\psi_0),~ \min_{y\in Y} \mu_y\right] >  0.
\end{equation}
\item[While] $\nr{G_y(\psi_k) - \mu_y}  \geq \eta$
  \begin{description}
\item[Step 1] Compute $d_{k} = - \D \Gglobal(\psi_k)^{+} (\Gglobal(\psi_k) - \mu)$
\item[Step 2] Determine the minimum $\ell \in \Nsp$ such that $\psi_k^\ell :=
  \psi_k + 2^{-\ell} d_k$ satisfies
\begin{equation*}
\left\{
\begin{aligned}
&\min_{y\in Y} G_y(\psi_k^\ell) \geq \eps_0 \\
&\nr{\Gglobal(\psi_k^\ell) - \mu} \leq (1-2^{-(\ell+1)}) \nr{\Gglobal(\psi_k) - \mu}
\end{aligned}
\right.
\end{equation*}
\item[Step 3] Set $\psi_{k+1} = \psi_k + 2^{-\ell}  d_k$ and $k\gets k+1$.
  \end{description}
\end{description}
\label{algo:newton}
\caption{Simple damped Newton's algorithm}
\end{algorithm}

The goal of this article is to prove the global convergence of this
damped Newton algorithm and to establish estimates on the speed of
convergence. As shown in Proposition~\ref{prop:newton}, the
convergence of Algorithm~\ref{algo:newton} depends on the regularity
and strong monotonicity of the map $\Gglobal = \nabla \Phi$. As we will see,
the regularity of $\Gglobal$ will depend mostly on the geometry of the cost
function and the regularity of the density. On the other hand, the
strong monotonicity of $\Gglobal$ will require a strong connectedness
assumption on the support of $\rho$, in the form of a weighted
Poincaré-Wirtinger inequality. Before stating our main theorem we give
some indication about these intermediate regularity and monotonicity
results and their assumptions.

\subsection{Regularity of Kantorovich's functional and MTW condition}
In order to establish the convergence of a damped Newton algorithm for
\eqref{eq:MA}, we need to study the $\Class^{2,\alpha}$ regularity of
Kantorovich's functional $\Phi$. However, while $\Class^1$ regularity
of $\Phi$ follows rather easily from the \eqref{eq:Tw} hypothesis (or
even from weaker hypothesis, see Theorem~\ref{th:Aurenhammerbis}),
higher order regularity seems to depend on the geometry of the cost
function in a more subtle manner. We found that a sufficient condition
for the regularity of $\Phi$ is the Ma-Trudinger-Wang condition
\cite{ma2005regularity}, which appeared naturally in the study of the
regularity of optimal transport maps. We use a  discretization
of Loeper's geometric reformulation of the Ma-Trundinger-Wang
condition \cite{loeper2009regularity}.

\begin{definition}[Loeper's condition]\label{def:Loeper}
  The cost $c$ satisfies Loeper's condition if for every $y$ in $Y$
  there exists a convex open subset $\Omega_y$ of $\Rsp^d$ and a $\Class^{1,1}$
  diffeomorphism $\exp^c_y: \Omega_y \to \Omega$ such that the functions
\begin{equation}
  p \in \Omega_y \mapsto c(\exp^c_y p, y) - c(\exp^c_y p,z) 
\tag{QC} \label{eq:MTW}
\end{equation}
are quasi-convex for all $z$ in $Y$. The map $\exp^c_y$ is called the
\emph{$c$-exponential} with respect to $y$, and the domain $\Omega_y$
is an \emph{exponential chart}.
\end{definition}
  
  We comment here, that when $Y$ is a finite subset of a continuous space and $c$ satisfies conditions \eqref{eq:Reg}, \eqref{eq:Tw}, and \eqref{A3w} (which can be found on pages \pageref{eq:Reg}, \pageref{eq:Tw}, and \pageref{A3w} respectively), the $c$-exponential map defined in the usual sense in optimal transport theory (see Remarks~\ref{rem:mtw} and \ref{rem:MTWremark}) will satisfy what we call  Loeper's condition above. However, it will become apparent that for our purposes what is essential is the above quasi-convexity property and not the actual definition of $\exp^c_y$. Thus we will elect to use the notation $\exp^c_y$ even in cases when $Y$ is not a finite subset of a continuous space. 

\begin{definition}[$c$-Convexity] Assuming Loeper's condition,
  a subset $X$ of $\Omega$ is  \emph{$c$-convex with respect to
      a point $y$ of $Y$} if its inverse image $(\exp_y^{c})^{-1}(X)$
    is convex. The subset $X$ is said to be \emph{$c$-convex} if it is
    \emph{$c$-convex} with respect to every point $y$ in $Y$.
  \end{definition}
Note that by assumption, the domain $\Omega$ itself is $c$-convex.
The connection between this discrete version of Loeper's condition and
the conditions used in the regularity theory for optimal transport is
detailed in Remark~\ref{rem:mtw}.  The \eqref{eq:MTW} condition
implies the convexity of each Laguerre cell in its own exponential
charts, namely $(\exp_y^{c})^{-1}(\Lag_y(\psi))$ is convex for
  every $y$ in $Y$.  This plays a crucial role in the regularity of
Kantorovich's functional.

\begin{theorem}
  \label{th:Regularity}
  Assume \eqref{eq:Reg}, \eqref{eq:Tw}, and \eqref{eq:MTW}. Let $X$ be a
  compact, $c$-convex subset of $\Omega$ and let $\rho$ in
  $\Probac(X) \cap \Class^{0,\alpha}(X)$ for $\alpha$ in $(0,1]$. Then,
    Kantorovich's functional is of class  $\Class^{2,\alpha}_\loc$ on the
    set
\begin{equation} \mathcal{K}^+ := \{ \psi: Y \to \Rsp \mid \forall y \in Y,~
  \rho(\Lag_y(\psi)) > 0 \},\label{eq:Kplus}
  \end{equation} and its Hessian is given by
\begin{align}
(z\neq y)\qquad&\frac{\partial^2 \Phi}{\partial \one_y \partial \one_z}(\psi) = \int_{\Lag_y(\psi) \cap \Lag_z(\psi)} \frac{\rho(x)}{\nr{D_x c(x,y) - D_x c(x,z)}_{g}} \dd\Haus_g^{d-1}(x),\notag\\
&\frac{\partial^2 \Phi}{\partial \one_y^2}(\psi) = -\sum_{z\in Y\setminus \{y\}}
\frac{\partial^2 \Phi}{\partial \one_y \partial \one_z}. \label{eq:Hess}
\end{align}
\end{theorem}
 The proof of this theorem and a more precise statement are given in
 Section~\ref{sec:Regularity} (Theorem~\ref{th:StrongRegularity}),
 showing that the $\Class^{2,\alpha}$ estimate can be made uniform
 when the mass of the Laguerre cells is bounded from below by a
 positive constant.

 \begin{remark}  \label{rem:mtw}
We remark that under certain assumptions on the cost $c$, our
\eqref{eq:MTW} condition is implied by classical conditions introduced
in a smooth setting by X.-N. Ma, N. Trudinger, and X.-J. Wang
\cite{ma2005regularity}, which include the well known (MTW) or (A3)
condition. See Remark \ref{rem:MTWremark} for more specifics.

There are a wide variety of known examples satisfying these
conditions. Aside from the canonical example of the inner product on
$\Rsp^n\times \Rsp^n$, and other costs on Euclidean spaces mentioned
in \cite{ma2005regularity, trudingerwang2009}, there are the nonflat
examples of Riemannian distance squared and
$-\log\nr{x-y}_{\Rsp^{n+1}}$ on (a subset of) $\mathbb{S}^n\times
\mathbb{S}^n$ (see \cite{loeper2011}). The last cost is associated to
the \emph{far-field reflector antenna problem}. We refer the reader to
\cite[p. 1331]{kimkitagawa14} for a (more) comprehensive list of such
costs.
\end{remark}



\subsection{Strong concavity of Kantorovich's functional}
As noted earlier, Kantorovich's functional $\Phi$ cannot be strictly
concave, since it is invariant under addition of a constant. This
implies that the Hessian $\D^2\Phi$ has a zero eigenvalue
corresponding to the constants. A more serious obstruction to the
strict concavity of $\Phi$ at a point $\psi$ arises when the discrete
graph induced by the Hessian (where two points are connected iff
$\partial^2\Phi/\partial \one_y \partial \one_z(\psi) \neq 0$) is not
connected. This can happen either because one of the Laguerre cells is
empty (hence not connected to any neighbor) or if the support of the
probability density $\rho$ is itself disconnected. In order to avoid
the latter phenomena, we will require that $(X,\rho)$ satisfies a
weighted $\LL^1$ Poincaré-Wirtinger inequality.

\begin{definition}[weighted Poincaré-Wirtinger]
  A continuous probability density $\rho$ on a compact set $X\subseteq
  \Omega$ satisfies a weighted Poincar\'e-Wirtinger inequality with
  constant $\Cpw>0$ if for every $\Class^1$ function $f$ on $X$,
\begin{equation}\label{eq:L1-poincare}
\nr{f-\Exp_\rho(f)}_{L^1(\rho)} \leq \Cpw \nr{\nabla f}_{L^1(\rho)},\tag{PW}
\end{equation}
where $\nr{h}_{L^1(\rho)}:=\int_X |h(x)|\rho(x)\d\Haus_g^d(x)$ and
$\Exp_\rho(f):=\int_X f(x)\rho(x)\d\Haus_g^d(x)$.
\end{definition}

We denote $E_Y$ the orthogonal complement (in $\Rsp^Y$) of the space
of constant functions on $Y$, that is $E_Y := \{ \psi \in \Rsp^Y \mid
\sum_y \psi(y) = 0 \}.$ As before, $\K^+$ is the set of functions
$\psi$ whose Laguerre cells all have positive mass.

\begin{theorem} Assume \eqref{eq:Reg}, \eqref{eq:Tw}, and \eqref{eq:MTW}. Let $X$ be a
  compact, $c$-convex subset of $\Omega$, and $\rho$ be a continuous
  probability density on $X$ satisfying \eqref{eq:L1-poincare}. Then,
  Kantorovich's functional $\Phi$ is strictly concave on $E_Y \cap
  \K^+$. \label{th:StrictConcavity}
\end{theorem}

As before, a more quantitative statement is proven in
Section~\ref{sec:Concavity} (Theorem~\ref{th:UniformConcavity}),
establishing strong concavity of $\Phi$ assuming that the mass of the Laguerre cells is  bounded from below by a positive
constant.

\subsection{Convergence result} \label{intro:convergence}
Putting Proposition~\ref{prop:newton}, Theorem~\ref{th:Regularity} and
Theorem~\ref{th:StrictConcavity} together, we can prove the global
convergence of the damped Newton algorithm for semi-discrete
optimal transport (Algorithm~1) together with optimal convergence
rates.

\begin{theorem} Assume \eqref{eq:Reg}, \eqref{eq:Tw} and \eqref{eq:MTW} and also that
  \begin{itemize}
  \item[(i)] The support of the probability density $\rho$ is included
    in a compact, $c$-convex subset $X$ of $\Omega$, and $\rho
    \in\Class^{0,\alpha}(X)$ for $\alpha$ in $(0,1]$.
  \item[(ii)] $\rho$ has positive Poincaré-Wirtinger
    \eqref{eq:L1-poincare} constant.
  \end{itemize}
  Then the damped Newton algorithm for semi-discrete optimal transport
  (Algorithm~1) converges globally with linear rate and locally with
  rate $1+\alpha$.
  \label{th:main}
\end{theorem}

\begin{remark}
  This theorem makes no assumption about the convexity (or
  $c$-convexity) of the support of the source density $\rho$. Such
  cases are not handled by other numerical methods for Monge-Ampère
  equations \cite{benamou2014numerical,loeper2011}.  For completeness,
  we provide in Appendix~\ref{app:poincare} an explicit example of a
  radial measure on $\Rsp^d$ whose support is an annulus but whose
  Poincar\'e-Wirtinger constant is nonetheless positive.
\end{remark}
\begin{remark}
  The positive lower bound on the damping parameter
  ($\tau_k=2^{-\ell}$ in Algorithm \ref{algo:newton}) established in
  this theorem degrades as $N$ grows to infinity. It is plausible (but
  far from direct) that one could control this quantity when $N$ is
  large by a comparison to the continuous Monge-Ampère equation. The
  strong concavity estimate (Theorem~\ref{th:StrictConcavity}) would
  then need to be replaced by uniform ellipticity estimates for the
  linearized Monge-Ampère equation, while the regularity estimate
  (Theorem~\ref{th:Regularity}) would be replaced by regularity
  estimates for solutions to the Monge-Ampère equation. We refer to
  Loeper and Rapetti \cite{loeper2005numerical} for an implementation
  of this ideas in a continuous setting. The space-discretization of
  their approach is open.
\end{remark}

\subsection*{Comparison to previous work.}
There exist a few other numerical methods relying on Newton's
algorithm for the resolution of the standard Monge-Ampère equation or
for the quadratic optimal transport problem. Here, we highlight some
of the differences between Algorithm~\ref{algo:newton} and
Theorem~\ref{th:main} and these existing results.  First, we note that
many authors have reported the good behavior in practice of Newton's or
quasi-Newton's methods for solving discretized Monge-Ampère equations
or optimal transport problems
\cite{merigot2011multiscale,de2012blue,benamou2014numerical}. Note
however that none of these works contain convergence proofs for the
Newton's algorithm.

Loeper and Rapetti \cite{loeper2005numerical} (refined by Saumier,
Agueh, and Khouider \cite{saumier2015efficient}) establish the global
convergence of a damped Newton's method for solving quadratic optimal
transport on the torus, relying heavily on Caffarelli's regularity
theory. In particular, the convergence of the algorithm requires a
positive lower bound on the probability densities, while this
condition is not necessary for Theorem~\ref{th:main} (see
Section~\ref{sec:Concavity} and Appendix~\ref{app:poincare} where we
construct explicitly probability densities with non-convex support
that still satisfy the hypothesis of Theorem~\ref{th:main}). A second
drawback on relying on the regularity theory for optimal transport is
that the damping parameter, which is an input parameter of the
algorithm used in \cite{loeper2005numerical}, cannot be determined
explicitly from the data. Third, the convergence proof is for
continuous densities, and it seems difficult to adapt it to the
space-discretized problem. On the positive side, it seems likely that
the convergence proof of
\cite{loeper2005numerical}\cite{saumier2015efficient} can be adapted
to cost functions satisfying the Ma-Trudinger-Wang condition (which is
equivalent to Loeper's condition \eqref{eq:MTW} that we also require).

Oliker and Prussner prove the \emph{local} convergence of Newton's
method for finding Alexandrov's solutions to the Monge-Ampère equation
$\det\D^2 u = \nu$ with Dirichlet boundary conditions, where $\nu$ is
a finitely supported measure \cite{oliker1989numerical}. Global
convergence for a damped Newton's algorithm is established by Mirebeau
\cite{mirebeau2015discretization} for a variant of Oliker and
Prussner's discretization, but without convergence
rates. Theorem~\ref{th:main} can be seen as an extension of the
strategy used by Mirebeau to optimal transport problems, which amounts
to (a) replacing the Dirichlet boundary conditions with the second
boundary value conditions from optimal transport (b) replacing the
Lebesgue measure by more general probability densities and (c)
changing the Monge-Ampère equation itself in order to deal with more
general cost functions.

We also comment here that our result Theorem~\ref{th:UniformConcavity}
answers a conjecture first raised by Gangbo and McCann, in the case
when the cost function satisfies the Ma-Trudinger-Wang condition. In
\cite[Example 1.6]{GangboMcCann1996}, a numerical approach to the
semi-discrete optimal transport problem is suggested by taking what is
equivalent to the negative gradient flow of the Kantorovich function
defined in \eqref{eq:Phi} above. There, Gangbo and McCann conjecture
that this gradient flow should converge, and our result of uniform
concavity of the Kantorovich functional provides a positive answer to
a quantitative strengthening of this conjecture, at least for costs,
measures, and domains satisfying the assumptions of
Theorem~\ref{th:UniformConcavity}.

Finally, we note that the overall strategy for proving the
  convergence of Algorithm~1 (proving regularity then strict concavity
  of $\Phi$) shares features to the one used in
  \cite{carlier2010knothe} to study the relationship between highly
  anisotropic semi-discrete quadratic optimal transport and Knothe
  rearrangement.

\subsection*{Outline}
In Section~\ref{sec:Aurenhammer}, we establish the differentiability
of Kantorovich's functional $\Phi$, adapting arguments from
\cite{aurenhammer1998minkowski}. In
Sections~\ref{section:localregularity} and \ref{sec:Regularity}, we
prove the (uniform) second-differentiability of Kantorovich's
functional when the cost function satisfies Loeper's \eqref{eq:MTW}
condition. Section \ref{sec:Concavity} is devoted to the proof of
uniform concavity of Kantorovich's functional, when the probability
density satisfies a Poincaré-Wirtinger inequality
\eqref{eq:L1-poincare}. In Section \ref{sec:Newton}, we combine these
intermediate results to prove the convergence of the damped Newton's
algorithm (Theorem~\ref{th:main}), and we present a numerical
illustration. Appendix~\ref{app:poincare} presents an explicit
construction of a probability density with non-convex support over
$\Rsp^d$ which satisfies the assumptions of
Theorem~\ref{th:main}. Appendix~\ref{app:proof-locreg} contains the
details of the proof of the main theorem of
Section~\ref{sec:Regularity}. 

\subsubsection*{Acknowledgements}
QM and BT would like to acknowledge the support of the French ANR
through the grant ANR-16-CE40-0014 (MAGA). BT is also partially
supported by LabEx PERSYVAL-Lab (ANR-11-LABX-0025-01).

\section{Kantorovich's functional}
\label{sec:Aurenhammer}
The purpose of this section is to present the variational formulation
introduced in \cite{aurenhammer1998minkowski} for the semi-discrete
optimal transport problem, adapting the arguments presented for the
squared Euclidean cost in \cite{aurenhammer1998minkowski} to cost
functions satisfying \eqref{eq:Regprime} and \eqref{eq:Twprime}, which
are weaker than the conditions \eqref{eq:Reg} and \eqref{eq:Tw}
presented in the introduction:

\begin{equation}
\forall y \in Y,~~ c(\cdot,y) \in \Class^{0}(\Omega) \tag{Reg'} \label{eq:Regprime}
\end{equation}
\begin{equation}
\forall y\neq z\in Y,\ \forall t\in \Rsp,~~ \Haus_{g}^d((c(\cdot,y) - c(\cdot,z))^{-1}(t)) = 0 \tag{Twist'} \label{eq:Twprime}
\end{equation}
Note that under \eqref{eq:Twprime}, the map $T_\phi: X \to Y$ defined
by \eqref{eq:Tpsi} is uniquely-defined $\Haus_{g}^d$--almost everywhere. Most of the results presented here are well known in the optimal transport literature, we include proofs for completeness.

\begin{theorem}
  \label{th:Aurenhammerbis}
  Assume \eqref{eq:Regprime} and \eqref{eq:Twprime}, and let $\rho$ be
  a bounded probability density on $X$ and $\nu = \sum_{y\in Y}
  \nu_y\delta_y$ be a probability measure over $Y$. Then, the
  functional $\Phi$ defined by \eqref{eq:Phi} is concave,
  $\Class^1$-smooth, and its gradient is given by \eqref{eq:GradPhi}.
\end{theorem}

The proof of Theorem~\ref{th:Aurenhammerbis} relies on
Propositions~\ref{prop:OT} and \ref{prop:Gcont}.

\begin{proposition} For any  $\psi:Y\to \Rsp$, the map $T_\psi$ is
  an optimal transport map for the cost $c$ between any probability
  density $\rho$ on $\Omega$ and the pushforward measure $\nu :=
T_{\psi \#} \rho$.
  \label{prop:OT}
\end{proposition}

\begin{proof} Assume that $\nu = S_{\#}\rho$ where $S$ is a
  measurable map between $X$ and $Y$. Then, by definition of $T_\psi$
  one has
$$ \forall x\in X,~~ c(x,T_\psi(x)) + \psi(T_\psi(x)) \leq c(x,S(x))
  + \psi(S(x)).$$ Multiplying this inequality by $\rho$ and
  integrating it over $X$ gives
$$ \int_{X} (c(x,T_\psi(x)) + \psi(T_\psi(x))) \rho(x) \dd\Haus_{g}^d(x)
  \leq \int_{X} (c(x,S(x)) + \psi(S(x))) \rho(x) \dd\Haus_{g}^d(x)$$ Since
  $\nu = S_{\#}\rho = T_{\psi\#} \rho$, the change of variable formula
  gives
$$ \int_{X}\psi(S(x)) \rho(x) \dd\Haus_{g}^d(x) = \int_Y \phi(y) \dd \nu =
  \int_X \psi(T_\psi(x)) \rho(x) \dd \Haus_{g}^d(x)$$ Substracting this
  equality from the inequality above shows that $T_\psi$ is optimal:
\begin{equation*} \int_X c(x,T_\psi(x)) \rho(x) \dd\Haus_{g}^{d}(x) \leq \int_X
  c(x,S(x)) \rho(x) \dd\Haus_{g}^{d}(x)\qedhere
  \end{equation*}
\end{proof}

\begin{proposition} \label{prop:Gcont}
  Assume \eqref{eq:Twprime} and \eqref{eq:Regprime}. Let $\rho$ be a
  probability density over a compact subset $X$ of $\Omega$. Then, the
  map $\Gglobal: \Rsp^Y \to \Rsp^Y$ is continuous:
\begin{equation}
 \Gglobal(\psi) = \left(\rho(\Lag_y(\psi))\right)_{y\in Y} 
\end{equation}
\end{proposition}

\begin{lemma} Let $\rho$ be a probability density over a compact subset $X$ of $\Omega$,  and let
  $f$ in $\Class^0(X)$ be such that $\rho(f^{-1}(t)) = 0$ for all $t
  \in \Rsp$. Then, the function $g: t \mapsto
  \rho(f^{-1}((-\infty,t]))$ is continuous.
    \label{lemma:leb}
\end{lemma}

\begin{proof} We consider the function $h(t) = \rho(f^{-1}((-\infty,t)))$. By hypothesis, $g(t) - h(t) = \rho(f^{-1}(t)) = 0$. Moreover, using Lebesgue's monotone convergence theorem one easily sees that $g$ (resp. $h$) is right-continuous (resp. left-continuous). This concludes the proof
\end{proof}

\begin{proof}[Proof of Proposition~\ref{prop:Gcont}]
  Proving the continuity of $\Gglobal$ amounts to proving the
  continuity of the functions $G_y(\psi) := \rho(\Lag_y(\psi))$ for
  any $y$ in $Y$. Fix $y$ in $Y$ and remark that by definition,
  $\Lag_y(\psi) = \bigcap_{z \neq y \in Y} H_z(\psi)$
  where $$H_z(\psi) := \{ x \in X\mid c(x,y) + \psi(y) \leq c(x,z) +
  \psi(z) \}.$$ Denoting $A\Delta B$ the symmetric difference of two
  sets, we have the following inequalities
\begin{equation}
  \abs{G_y(\psi) - G_y(\phi)} \leq
  \rho(\Lag_y(\psi) \Delta  \Lag_y(\phi)) \leq \sum_{z\in Y\setminus\{y\}}  \rho(H_z(\psi)\Delta H_z(\phi)).
  \label{eq:symdiff:bnd}
  \end{equation}
Fix $z \neq y \in Y$, and denote $f = c(\cdot,y) - c(\cdot,z)$. Then,
$$ H_z(\psi)\Delta H_z(\phi) \subseteq f^{-1}([\psi(z) - \psi(y),
  \phi(z) - \phi(y)]). $$ 
Here and after, we use the convention  that $[a,b]=[\min\{a,b\},\max\{a,b\}]$. 
  By \eqref{eq:Twprime} and
Lemma~\ref{lemma:leb} we know that
$\lim_{\phi \to \psi} \rho(H_z(\psi)\Delta H_z(\phi)) = 0,$ which
with \eqref{eq:symdiff:bnd} concludes the proof.
\end{proof}

\subsection{Proof of Theorem~\ref{th:Aurenhammer}} We simultaneously  show that
  the functional is concave and compute its gradient.  For any
  function $\psi$ on $Y$ and any measurable map $T: X\to Y$, one has
  $\min_{y\in Y} (c(x,y) + \psi(y)) \leq c(y,T(y)) + \psi(T(y))$, which
  by integration gives
\begin{equation}
 \Phi(\psi) \leq \int_{X} (c(x,T(x)) + \psi(T(x)))\rho(x)\dd\Haus_{g}^{d}(x) - \sum_{y\in Y}\psi(y)\nu_y. 
\label{eq:AurSD}
\end{equation}
Moreover, equality holds when $T=T_\psi$. Taking another function
$\phi$ on $Y$ and setting $T = T_\phi$ in Equation~\eqref{eq:AurSD}
gives $$\Phi(\psi) \leq \Phi(\phi) +\sca{\Gglobal(\phi) - \nu}{\psi-\phi},$$
where $\Gglobal$ is defined as in the statement of 
Proposition~\ref{prop:Gcont}. This proves that the
superdifferential $\partial^+\Phi(\phi)$  contains $G(\phi)-\nu$, thus establishing the
concavity of $\Phi$ and its differentiability almost everywhere. It is known that the 
supergradient $\partial^+ \Phi (\phi)$ is characterized by~\cite[Theorem 25.6]{rockafellar1970convex}
$$
\partial^+ \Phi (\phi)=\mathrm{conv}\left\{\lim_{n\to\infty} \nabla \Phi(\phi_n)\mid (\phi_n)\in S\right\},
$$
where $\mathrm{conv}$ denotes the convex envelope and $S$ the set of sequences $(\phi_n)$ converging to $\phi$ such that $\Phi$ is differentiable at $\phi_n$.
By Proposition~\ref{prop:Gcont}, the map $\Gglobal$ is continuous,
meaning that we have constructed a continuous selection of a
supergradient in the superdifferential of the concave function
$\Phi$
$$
\partial^+ \Phi (\phi)
= \mathrm{conv}\left\{\lim_{n\to\infty} \nabla \Phi(\phi_n)\right\}
= \mathrm{conv}\left\{\lim_{n\to\infty} G(\phi_n)-\nu\right\}
= \{ G(\phi)-\nu \}.
$$
This proves that $\Phi$ is $\Class^1$, and that $\nabla \Phi(\phi) = \Gglobal(\phi) - \nu$.

\section{Local regularity in a  $c$-exponential chart}\label{section:localregularity}
The results presented in this section constitute an intermediate step
in the proof of $\Class^{2,\alpha}$ regularity of Kantorovich's
functional. Let $\Xloc$ be a compact, convex subset of $\Rsp^d$ and
$f_1,\hdots, f_N$ be $\Class^{1,1}$ functions on $\Xloc$ which are
quasi-convex, meaning that for any scalar $\lambda \in \Rsp$ the closed
sublevel sets $K_i(\lambda) := f_i^{-1}([-\infty, \lambda])$ are convex.
Let $\rholoc$ be a continuous probability density over $\Xloc$. The
purpose of this section is to give sufficient conditions to ensure
the regularity of the following function $\Gloc$ near the origin of $\Rsp^N$:
\begin{align}
  &\Gloc: \bm{\lambda} \in \Rsp^N \mapsto \int_{K(\bm{\lambda})} \rholoc(x) \d\Haus^d(x), \label{eq:G}\\
  &\hbox{where }
  K(\bm{\lambda}) := 
  \bigcap_{i=1}^N K_i(\lambda_i) 
  = \{ x \in \Xloc \mid \forall i \in \{1,\hdots,N\},~ f_i(x) \leq \lambda_i \}. \notag
\end{align}

\subsection{Assumptions and statement of the theorem.} We will impose
two conditions on the functions $(f_i)_{1\leq i\leq N}$. As we will
see in Section~\ref{sec:Regularity}, both conditions are satisfied
when these functions $(f_i)$ are constructed from a semi-discrete
optimal transport transport problem whose cost function satisfy
Loeper's condition (see Definition~\ref{def:Loeper}).

\subsubsection*{Non-degeneracy} The functions $(f_i)$
satisfy the \emph{nondegeneracy condition} if the norm of their
gradients is bounded from below: 
\begin{equation} \eps_{nd} := \min_{1\leq i\leq N} \min_{\Xloc} \nr{\nabla f_i} > 0. \tag{ND} 
\label{eq:ND}
\end{equation}
This condition is necessary for the continuity of the map $\Gloc$ even
when $N=1$.

\subsubsection*{Transversality}
The boundary of the convex set $K(\bm{\lambda})$ can be decomposed
into $N+1$ \emph{facets}, namely $(K(\bm{\lambda}) \cap \partial
K_i(\lambda_i))_{1\leq i\leq N}$ and $K(\bm{\lambda})\cap \partial
\Xloc$. The purpose of the transversality condition we consider is to
ensure that the angle between adjacent facets is bounded from below 
when $\bm{\lambda}$ remains close to some fixed vector $\bm{\lambda}_0$.

\begin{definition}[Normal cone] Let $K$ be a convex compact set of $\Rsp^d$.  The \emph{normal cone}
to $K$ at a point $x$ in $K$ is the set
\begin{equation}
  \Normal_x K=\{v\in\Rsp^d\mid \forall y \in K,~\sca{y-x}{v}\leq 0 \},
  \label{eq:Nor}
\end{equation}
and its elements are said to be \emph{normal to $K$ at $x$}.
\end{definition}

\begin{definition}[Transversality] The family of functions $(f_i)$ satisfy the
\emph{transversality condition near $\bm{\lambda}_0$} if there exists
positive constants $\eps_{tr}$ and $T_{tr}\leq 1$ such that for every
$\bm{\lambda}$ in $\Rsp^N$ satisfying $\nr{\bm{\lambda} -
  \bm{\lambda_0}}_{\infty} \leq T_{tr}$ for the usual $\ell^\infty$
norm on $\Rsp^N$ and every point $x$ in $\partial K(\bm{\lambda})$ one
has,
\begin{equation}
        \tag{T}
\label{eq:T}
\begin{aligned}
  &\hbox{ if } \exists i \neq j \in\{1,\hdots,N\} \hbox{ s.t. } f_i(x) = \lambda_i \hbox{ and } f_j(x) = \lambda_j, \\
  &\hspace{4cm} \hbox{ then, }
    \left(\frac{\sca{\nabla f_i(x)}{\nabla f_j(x)}}{\nr{\nabla f_i(x)}\nr{\nabla f_j(x)}}\right)^2 \leq 1-\eps^2_{tr} \\
    &\hbox{ if } \exists i  \in\{1,\hdots,N\} \hbox{ s.t. } f_i(x) = \lambda_i \hbox{ and } x\in \partial \Xloc,\\
    &\hspace{3.5cm}\hbox{ then, }
   \forall u\in \Normal_x \Xloc,~ \left(\frac{\sca{u}{\nabla f_j(x)}}{\nr{u}\nr{\nabla f_j(x)}}\right)^2 \leq 1-\eps^2_{tr},
    \end{aligned}
\end{equation}
Note that $\partial \Xloc$ is smooth at $x$, $\Normal_x \Xloc$ is the
ray spanned by the exterior normal to $\Xloc$ at $x$.
\end{definition}

\begin{theorem}
  \label{th:LocReg} Assume that the functions satisfy the non-degeneracy condition
  \eqref{eq:ND} and the transversality condition \eqref{eq:T} near
  $\bm{\lambda_0}$. Let $\rholoc$ be a $\Class^{0,\alpha}$ probability
  density on $\Xloc$.  Then, the map $\Gloc$ defined in \eqref{eq:G}
  is of class  $\Class^{1,\alpha}$ on the cube $Q := \bm{\lambda_0} +
  [-T_{tr},T_{tr}]^N$ and has partial derivatives given by
\begin{equation}
  \frac{\partial \Gloc}{\partial \lambda_i}(\bm{\lambda}) =
  \int_{K({\bm{\lambda}}) \cap \partial K_i({\bm{\lambda}})}
  \frac{\rholoc(x)}{\nr{\nabla f_i(x)}} \dd \Haus^{d-1}(x). \label{eq:PD}
\end{equation}
In addition, the norm $\nr{\Gloc}_{\Class^{1,\alpha}(Q)}$ is bounded
by a constant depending only on $\eps_{tr},\eps_{nd},
\nr{\rholoc}_{\Class^{0,\alpha}(\Xloc)}$, on the diameter of $\Xloc$ and on
$$ C_M := \max_{1\leq i\leq N} \nr{\nabla f_i}_\infty \qquad C_L := \max_{1\leq i\leq N}
\nr{\nabla f_i}_{\Lip(\Xloc)}.$$
\end{theorem}

Note that the $\Class^{1,\alpha}$ constant of $\Gloc$ depends on the
transversality constant~$\eps_{tr}$ but that it does not depend on
$T_{tr}$.

\subsection{Sketch of proof}
The correct expression for the partial derivatives of $\Gloc$, given
by equation \eqref{eq:PD}, can easily be guessed by applying the
co-area formula. The non-degeneracy condition then ensures that the
denominator in this expression does not vanish.  What is more delicate
is to prove that these partial derivatives are $\alpha$-H\"older, with
a uniform estimate on the $\alpha$-H\"older norm. A second application
of the co-area formula on the manifold $f_{i}^{-1}(\lambda_i)$
suggests that for $j\neq i$ one should have,
$$ \abs{\frac{\partial}{\partial \lambda_j} \int_{K({\bm{\lambda}})
    \cap \partial K_i({\bm{\lambda}})} \frac{\rholoc(x)}{\nr{\nabla
      f_i(x)}} \dd \Haus^{d-1}(x) } \leq C
\Haus^{d-2}(K({\bm{\lambda}}) \cap \partial K_i({\bm{\lambda}}) \cap
\partial K_j({\bm{\lambda}})) , $$ under the assumption that the
density $\rholoc$ is $\Class^1$ and the facet $K({\bm{\lambda}}) \cap
\partial K_i({\bm{\lambda}})$ does not intersect $\partial \Xloc$. It will turn out that, thanks to the transversality hypothesis, the
$\Haus^{d-2}$-measure of the union $\Sigma(\bm{\lambda})$ of these facets can be bounded
uniformly:
\begin{equation*}
  \Sigma(\bm{\lambda}) = \bigcup_{1\leq i\leq N} (K(\bm{\lambda}) \cap \partial \Xloc \cap \partial K_i(\lambda_i)) \cup
  \bigcup_{1 \leq i < j \leq N}
 (K(\bm{\lambda}) \cap \partial K_{i}(\lambda_{i}) \cap  \partial K_{j}(\lambda_{j})).
\end{equation*}
Note also that equivalently, a point $x$ belongs to the singular set
$\Sigma(\bm{\lambda})$ if and only if it satisfies one of the
assumptions in \eqref{eq:T}.  In the next subsection, we prove an upper
bound on $\Haus^{d-2}(\Sigma(\bm{\lambda}))$ (see
Proposition~\ref{prop:HausSigma}). The proof of Theorem~\ref{th:LocReg}
follows from this upper bound and from several applications of the
co-area formula. Since it is elementary but quite long, we have
postponed the proof of the theorem itself to Appendix~\ref{app:proof-locreg}.

\subsection{A control on the $(d-2)$--Hausdorff measure of singular points}
In this section, we prove that the transversality condition
\eqref{eq:T} and the quasi-convexity of the functions $(f_i)$ imply a
uniform upper bound on the $(d-2)$--Hausdorff measure of
$\Sigma(\bm{\lambda})$.

\begin{proposition} \label{prop:HausSigma}
  Assuming the transversality condition \eqref{eq:T}, there exists a
  constant depending only on $d$ and $\diam(\Xloc)$ such that for
  every $\nr{\bm{\lambda}}_\infty \leq T_{tr}$,
$$ \Haus^{d-2}(\Sigma(\bm{\lambda})) \leq C(d,\diam(\Xloc)) \cdot
  \frac{1}{\eps_{tr}}.$$
\end{proposition}

We will deduce this proposition from a general upper bound on the
$(d-2)$--Hausdorff measure of the set of $\tau$--singular points of a
compact convex body. A more general and quantitative version of this
bound can be found in \cite{hug1998generalized}. We provide below a
straightforward and easy proof based on the notions of packing and
covering numbers.

\begin{proposition}\label{prop:hug}
  Let $K$ be a convex, compact set of $\Rsp^d$ and $\tau > 0$. Then, 
$$
\Haus^{d-2}(\Sing(K,\tau)) \leq C(d,\diam(K))\ \frac{1}{\tau},
$$
where $\Sing(K,\tau):=\{x\in \partial K \mid \exists u,v \in \Normal_{x}(K)\cap \Sph^{d-1},~~ \sca{u}{v}^2 \leq 1- \tau^2 \}.$
\end{proposition}

Recall that the covering number $\Cov(K,\eta)$ of a subset $K\subseteq
\Rsp^d$ is the minimum number of Euclidean balls of radius $\eta$
required to cover $K$. The packing number of a subset $K$ is given by
$$ 
\Pack(K,\eta) := \max \{\mathrm{Card}(X) \mid X\subseteq K \hbox{ and }
\forall x\neq y\in X, \nr{x-y}\geq \eta \}.
$$
We will use the following comparisons between covering and packing numbers:
\begin{equation}\label{eq:covering_packing}
\Cov(K,\eta) \leq  \Pack(K,\eta) \leq \Cov(K,\frac{\eta}{2}).
\end{equation}

\begin{proof}[Proof of Proposition~\ref{prop:hug}]
The proof consists in comparing a lower bound and an upper bound of
the packing number of the set
$$
U :=\{(x,n)\in \Rsp^d\times \Sph^{d-1} \mid x\in \Sing(K,\tau)\mbox{ and } n\in \Normal_x(K)\}.
$$

\noindent \textbf{Step 1.} We first calculate an upper bound on the
covering number of the unit bundle $\Unit K := \{ (x,n) \in \partial
K\times \Sph^{d-1} \mid n \in \Normal_x K \}$. Given a positive radius
$r$, we denote by $K^r$ the set of points that are within distance $r$ of
$K$. By convexity, the projection map $\p_K:\Rsp^d\to K$, mapping a
point to its orthogonal projection on $K$, is well defined and
$1$-Lipschitz. We consider
\begin{align*}
\pi : \partial K^r & \to  \mathcal{U}(K)  \\
x&\mapsto  \left(\p_K(x), \frac{x-\p_K(x)}{\nr{x-\p_K(x)}}\right)
\end{align*}
The map $\pi$ is surjective and has Lipschitz constant
$L:=\sqrt{1+4/r^2}$. We deduce an upper bound on covering number of
$\Unit K$ from the covering number of the level set $\partial K^r$:
\begin{equation*}
\Cov(\Unit(K),\varepsilon)
\leq  \Cov\left(\partial K^r, \frac{\varepsilon}{L}\right).
\end{equation*}
Now, consider a sphere $S$ with diameter $2 \diam(K)$ that encloses
the tubular neighborhood $K^r$ with $r:= \diam(K)$. The projection map
$\p_{K^r}$ is $1$-Lipschitz, and $\p_{K^r}(S) = \partial K^r$. Using
the same argument as above, we have:
$$ \Cov\left(\partial K^r, \eta\right) \leq \Cov(S, \eta) \leq C(d) \cdot (\diam(K)/\eta)^{d-1}.$$
Combining these  bounds with the inclusion $U \subseteq \Unit(K)$ gives us
\begin{equation}\label{eq:hug1}
  \Cov(U,\varepsilon) \leq \frac{C(d,\diam(K))}{\eps^{d-1}}.
  \end{equation}
\noindent \textbf{Step 2.} We now establish a lower bound for
$\Pack(U,2\varepsilon)$. Let $x$ be a $\tau$-singular point and $u,v$ be
two unit vectors such that $\sca{u}{v}^2 \leq 1 - \tau^2$. This
implies that $\Normal_x K\cap
\Sph^{d-1}$ contains a spherical geodesic 
segment of length at least $C \cdot \tau$, giving us a lower bound on
the packing number of $\Normal_x K\cap
\Sph^{d-1}$, namely $\Pack(\Normal_x K\cap
\Sph^{d-1}, \eta)
\geq C \cdot \tau/\eta$. Now, let $X$ be a maximal set in the
definition of the packing number $\Pack(\Sing(K,\tau),2\varepsilon)$
and for every $x\in X$, let $Y_x$ be a maximal set in the definition
of the packing number $\Pack(\Normal_x(K)\cap
\Sph^{d-1},2\varepsilon)$, so that $\Card(Y_x) \geq C \cdot
\tau/\eps$. Then, the set $Z := \{ (x, y) \mid x \in X,~ y \in Y_x \}$
is a $2\varepsilon$ packing of $U$, and the cardinality of this set is
bounded from below by $C \cdot \Card(X) \cdot \tau/\eps$. This
gives
\begin{equation}\label{eq:hug2}
\Pack(U,2\varepsilon) \geq  C \cdot \Pack(\Sing(K,\tau), 2\eps) \cdot \tau/\eps.
\end{equation}

\noindent \textbf{Step 3.}
Combining Equations \eqref{eq:hug1}, \eqref{eq:hug2} and the
comparison between packing and covering numbers
\eqref{eq:covering_packing}, we get
\begin{align*} \Pack(\Sing(K,\tau), 2\eps)
  &\leq \frac{C(d,\diam(K))}{\tau \eps^{d-2}}.
\end{align*}
Using the comparison between packing and covering numbers, this means
that we can cover $\Sing(K,\tau)$ with $N_\eps$ balls of radius
$\eps$, such that $N_\eps \leq C(d,\diam(K))/(\tau
\eps^{d-2})$. By definition of the Hausdorff measure, we have
\begin{equation*}
  \Haus^{d-2}(\Sing(K,\tau)) \leq \liminf_{\eps\to 0} N_\eps \eps^{d-2} \leq C(d,\diam(K)) \frac{1}{\tau}. \qedhere
  \end{equation*}
\end{proof}

\begin{proof}[Proof of Proposition~\ref{prop:HausSigma}]
  Given $\nr{\bm{\lambda}}_\infty \leq
  T_{tr}$, the  transversality condition \eqref{eq:T} implies
  $$ \forall x\in \Sigma(\bm{\lambda}),~~\exists u,v \in \Normal_x
  K(\bm{\lambda}),~~ \left(\frac{\sca{u}{v}}{\nr{u}\nr{v}}\right)^2
  \leq 1 -\eps_{tr}^2, $$ where $\Normal_x K(\bm{\lambda})$ is the
  normal cone to the convex set $K(\bm{\lambda})$ at $x$ (see
  \eqref{eq:Nor}). This implies that $\Sigma(\bm{\lambda})$ is
  included in the set $\Sing(K(\bm{\lambda}), \eps_{tr})$ of
  $\tau$-singular points 
  with $\tau =
  \eps_{tr}$. The conclusion then follows from
  Proposition~\ref{prop:hug}.
\end{proof}


\section{$\Class^{2,\alpha}$ regularity of Kantorovich's functional}
\label{sec:Regularity}
This section is devoted to the proof of  the following regularity result. Recall that the conditions \eqref{eq:Reg}, \eqref{eq:Tw}, and \eqref{eq:MTW} are defined in the introduction, respectively on pages \pageref{eq:Reg}, \pageref{eq:Tw}, and \pageref{eq:MTW}.

\begin{theorem}
  \label{th:StrongRegularity}
  Assume \eqref{eq:Reg}, \eqref{eq:Tw}, and \eqref{eq:MTW}. Let $X$ be
  a compact, $c$-convex subset of $\Omega$ and $\rho$ in
  $\Probac(X) \cap \Class^{0,\alpha}(X)$ for $\alpha$ in $(0,1]$. Then,
   the Kantorovich's functional $\Phi$ is uniformly $\Class^{2,\alpha}$ on the
    set
\begin{equation} \mathcal{K}^\eps := \{ \psi: Y \to \Rsp \mid \forall y \in Y,~
  \rho(\Lag_y(\psi)) > \eps \} \label{eq:Keps}
  \end{equation}
and its Hessian is given by \eqref{eq:Hess}. In addition,
the $\Class^{2,\alpha}$ norm of the restriction of $\Phi$ to
$\mathcal{K}^\eps$ depends only on $\nr{\rho}_\infty$, $\eps$, $\diam(X)$,
and the constants defined in Remark~\ref{rem:Constants} below.
\end{theorem}

For the remainder of the section, for any point $y$ in $Y$, we will
denote $X_y = (\exp_y^c)^{-1}(X) \subseteq \Rsp^d$ the inverse image of
the domain $X$ in the exponential chart at $y$. The set $X_y$ is
convex by $c$-concavity of $X$. We consider the functions
$$f_{z,y}: p\in X_y \mapsto c(\exp_y^c(p), y) - c(\exp_y^c(p), z),$$
which are quasi-concave by \eqref{eq:MTW}. The main difficulty in deducing Theorem~\ref{th:StrongRegularity} from
Theorem~\ref{th:LocReg} is in establishing the quantitative
transversality condition \eqref{eq:T} introduced on p.\,\pageref{eq:T} for the family of functions
$(f_{z,y})_{z\in Y\setminus\{y\}}$.

\begin{remark}[Constants] \label{rem:Constants} The $\Class^{2,\alpha}$
  norm of the restriction of $\Phi$ to $\mathcal{K}^\eps$ explicitely
  depends on the following constants, whose finiteness (or positivity)
  follows from the compactness of the domain $X$, from the finiteness
  of the set $Y$ and from the conditions \eqref{eq:Reg},
  \eqref{eq:Tw}, and \eqref{eq:MTW}:
  \begin{align}
 \eps_{tw}&:=\min_{x\in X} \min_{y,z\in Y \atop {y\neq z}}
 \nr{D_xc(x,y) - D_c(x,z)}_{g} > 0 \notag\\ 
 C_\nabla &:=
 \max_{(x,y)\in X\times Y} \nr{D_x c(x,y)}_{g} < +\infty \notag\\ 
 \biLip
 &:= \max_{y\in Y} \max\left\{\nr{\exp_y^c}_{\Lip(X_y)},
 \nr{(\exp_y^c)^{-1}}_{\Lip(X)}\right\} < +\infty,\label{eq:constants}
  \end{align}
where we recall that 
$X_y := \exp_{y}^{-1}(X)$.
Our estimates will also rely on the following constants involving the
differential of the exponential maps. As before, the tangent spaces
$\Tang_x \Omega$ are endowed with the Riemannian metric $g$ from
$\Omega$. We set:
\begin{align*}
\Ccond &:= \max_{y \in Y} \max_{p \in X_y}
\cond(\restr{\D \exp_y^c}_{p}),\\
\Cdet &:= \max_{y\in Y} \nr{\det(\D\exp_y^c)}_{\Lip(X_y)},
\end{align*}
where $\cond(A)$ is the condition number of a linear transform
$A$ between finite dimensional normed spaces and $\det(A)$ is the determinant of $A$ with
respect to orthornormal bases. The quantitative transversality
estimates involve all the above constants in an explicit way, see \eqref{eq:EpsTRConstants}.
\end{remark}

\begin{remark}
  Even in the Euclidean case, one needs a lower bound on the volume of
  Laguerre cells in order to establish the second-differentiability of
  the functional $\Phi$. Indeed, let $y_\pm = \pm 1$, let $y_0 = 0$,
  $Y = \{ y_{-}, y_0, y_+ \} \subseteq \Rsp$. Consider the cost
  $c(x,y) = -xy$, and the density $\rho = 1$ on
  $X=[-\frac{1}{2},\frac{1}{2}]$. Let $\phi_\tau \in \Rsp^Y$ be
  defined by $\phi(y_{\pm}) = 1/2$ and $\phi(y_0) = \tau$. A simple
  calculation gives, for $\tau \geq 0$,
  \begin{align*}
    \frac{\partial\Phi}{\partial
      \one_{y_0}}(\psi_\tau) = \max(1-2\tau,0), 
  \end{align*}
  which is not differentiable at $\tau = 1/2$, even though \eqref{eq:Reg},
  \eqref{eq:Tw}, and \eqref{eq:MTW} are all satisfied.
\end{remark}

\noindent\textbf{Outline.}  In Section~\ref{subsec:lower_estimate}, we
establish a part of the transversality condition using elementary
properties of convex sets
(Proposition~\ref{prop:explicitlowerbound}). We  establish 
in Section \ref{subsec:upper-estimate} a second transversality condition 
using additional assumptions and proceed in
Section \ref{subsec:ProofRegularity} to the proof of Theorem
\ref{th:StrongRegularity}. In Section
\ref{subsec:alternative-bound}, we propose an alternative
transversality estimate when $Y$ is a sample subset of a target domain
$\Omega'$ (Proposition \ref{prop:explicitupperbound2}).

\subsection{Lower transversality estimates}\label{subsec:lower_estimate}
Next, we undertake a series of proofs to obtain explicit constants in
the transversality estimate \eqref{eq:T}, which depend on the choices
of cost, domains, and dimension. Consider the Laguerre cell of a point
$y$ in $Y$ in its own exponential chart, that is
$$L_y(\psi) := (\exp_y^c)^{-1}(\Lag_y(\psi)) = \{ p\in X_y \mid
f_{z,y}(p) \leq \psi(z) - \psi(y) \}.$$ The set $L_y(\psi)$ is the
intersection of sublevel sets of the functions $f_{z,y}$, and is
therefore a convex subset of $X_y$ by condition \eqref{eq:MTW}.
The first proposition establishes that two unit outer normals to $L_y$
with the same basepoint cannot be near-opposite. Recall the definition
of the normal cone from \eqref{eq:Nor}.

\begin{proposition}\label{prop:explicitlowerbound}
  Assume that $\psi$ lies in $\mathcal{K}^{\eps/2}$ (see
  \eqref{eq:Keps}). For any $y$ in $Y$, any point $p$ in $\partial
  L_y(\psi)$ and any unit normal vectors $v,w \in \Normal_p L_y(\psi)$
  one has
  \begin{align}\label{eq:awayfromopposite}
    \sca{v}{w} \geq -1+ \delta_0^2,
  \end{align}
  where $\delta_0 := \eps/(2^{d-1}\nr{\rho}_\infty{\biLip}^{2d}\diam{(X)}^{d}) \leq 1$.
\end{proposition}
The proof of this proposition follows from a general lemma about
convex sets. By convexity \eqref{eq:MTW}, the set $L_y(\psi)$ is
contained in an intersection of two half-spaces with outward normals
$v$ and $w$ at $p$, giving an upper bound on its volume in term of its
diameter and the angle between $v$ and $w$ (see
Figure~\ref{fig:lbvw}). On the other hand, we know that the volume of
$L_y(\psi)$ is bounded from below by a constant depending on $\eps$. Comparing
these bounds will give us the one-sided estimate
\eqref{eq:awayfromopposite}.

\begin{lemma} \label{lemma:lbvw}
  Let $K$ be a bounded convex set of $\Rsp^d$, let $p$ be a boundary
  point of $K$ and $v,w$ be two unit (outward) normal vectors to $K$
  at $p$. Then,
  $$ -1 + \delta^2_K \leq \sca{v}{w} \hbox{ where } \delta_K =
  \frac{\Haus^{d}(K)}{2^{d-2}\diam(K)^{d}} \leq 1.$$
\end{lemma}

\begin{figure}
  \includegraphics{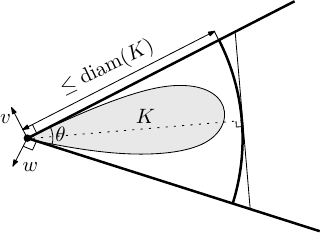}
  \caption{\label{fig:lbvw} Bound on the volume of a convex set $K$ as
    a function of the angle between two normal vectors $v,w$ at the
    same point and its diameter (see Lemma~\ref{lemma:lbvw}).}
\end{figure}

\begin{proof}
  The left-hand side of the inequality is non-positive, so the
  inequality needs only to be proven when $\sca{v}{w}\leq 0$, which we
  assume from now on. Making a rotation of axes and a translation if
  necessary, we assume that $p$ lies at the orign and that the unit
  vectors span the first two coordinates of $\Rsp^d$.  Then, letting
  $H :=\{p\mid \sca{p}{v}\leq 0\}$, $H':= \{p\mid \sca{p}{w}\leq 0\}$
  and $D$ be the two-dimensional disc centered at $0$ of radius
  $\diam(K)$, one has
\begin{equation*}
K \subseteq H\cap H'\cap (D\times [-\diam(K), \diam(K)]^{d-2}).
\end{equation*}
The intersection $H \cap H' \cap D$ is an angular sector of the disc
$D$, whose angle is equal to $\theta := \pi -
\arccos(\sca{v}{w})$ (see Figure~\ref{fig:lbvw}). Therefore, we have
\begin{align}
  \Haus^d(K) &\leq \Haus^d(H\cap H'\cap (D\times [-\diam(K), \diam(K)]^{d-2}))  \notag \\
  &\leq   2^{d-2}\diam(K)^d\tan(\theta/2). \label{ineq:lbvw1}
\end{align}
Using the expression of $\cos(\theta)$ in term of $\tan(\theta/2)$ and
recalling  $\sca{v}{w}\leq 0$, 
\begin{equation}
  \tan(\theta/2) = \sqrt{\frac{1+\sca{v}{w}}{1-\sca{v}{w}}} \leq \sqrt{1+\sca{v}{w}}
  \label{ineq:lbvw2}
\end{equation}
The lemma follows directly from Equations
\eqref{ineq:lbvw1}--\eqref{ineq:lbvw2}.
\end{proof}

\begin{proof}[Proof of Proposition~\ref{prop:explicitlowerbound}] 
  By definition of the bi-Lipschitz constant $\biLip$, 
  \begin{align*}
    \Haus^d(L_y(\psi)) \geq \eps/(2 \biLip^d\nr{\rho}_\infty)\hbox{ and }
    \diam(L_y(\psi)) \leq \biLip \diam(X).
  \end{align*}
  Applying the above lemma to the two outward normals $v,w$ at $p$, we get
  \begin{equation*}
    \sca{v}{w} + 1 \geq  \frac{\Haus^{d}(L_y(\psi))^2}{4^{d-2}\diam(L_y(\psi))^{2d}} \geq
    \frac{\eps^2}{4^{d-1}\biLip^{4d}\nr{\rho}_\infty^2\diam(X)^{2d}}.\qedhere
    \end{equation*}
\end{proof}

We also record the following lemma for later use.
\begin{lemma}
  Let $y$ in $Y$ and let $p$ in be a point of $L_y(\psi)$ such that
  for some $z\neq y$,~ $f_{z,y}(p) = \psi(z) - \psi(y)$. Then, the
  point $p' :=(\exp_z^c)^{-1}(\exp_y^c(p))$ belongs to $L_z(\psi)$ and the vector
  $\nabla f_{y,z}(p)$ lies in the normal cone $\Normal_{p'}
  L_z(\psi)$. \label{lemma:changeoflag}
 \end{lemma}
 \begin{proof}
   We introduce the point $x = \exp_y^c(p)$.  The hypothesis is
   equivalent to $c(x,y) + \psi(y) = c(x,z) + \psi(z).$ Since $p$
   belongs to $L_y(\psi)$, the point $x$ belongs to $\Lag_y(\psi)$. Then, for any $z'\in Y$,
   \begin{align*}
    c(x,z) + \psi(z)&= c(x,y) + \psi(y) \leq c(x,z') + \psi(z')\\
   \end{align*}
   thus establishing that $x$ belongs to $\Lag_z(\psi)$ or equivalently $p'\in L_z(\psi)$.
 \end{proof}

\subsection{Upper transversality estimates.}\label{subsec:upper-estimate}
We now turn to the proof of the quantitative transversality
estimates. We begin with a bound which involves the condition
number of differential of exponential maps, see
Remark~\ref{rem:Constants}.  The advantage of this first bound is that
we do not have to assume that the points in $Y$ are sampled from a
continuous domain. A second transversality estimate is presented in
\S\ref{subsec:regularitymtw}. 

\subsubsection*{Notation.}
We introduce notation that will be used in throughout this section. We
fix a point $y_0$ in $Y$ and also fix an arbitrary ordering of the
remaining points, so that  $Y = \{y_0, y_1,\hdots,y_{N}\}$. We
define $\Xloc := X_{y_0}$ and for every index $i \in\{1,\hdots,N\}$ we put
\begin{align*}
  f_i:=f_{y_i,y_0}: p \in \Xloc &\mapsto  c(\exp_{y_0}^c(p), y_0) - c(\exp_{y_0}^c(p), y_i).
\end{align*}
By the \eqref{eq:Tw} condition, these functions $f_1,\hdots,f_N$
satisfy the non-degeneracy condition \eqref{eq:ND}, and we have the
following inequalities:
\begin{align}
  &\eps_{nd}:=\min_{i,j \neq 0} \min_{p \in X_{y_0}} \nr{\nabla
    f_i(p) -\nabla f_j(p) } \geq \biLip^{-1} \eps_{tw} > 0,\label{eq:gradlip}\\
\label{eq:gradlength}
&\sup_{i \neq 0} \sup_{p\in X_{y_0}} \nr{\nabla f_{i}(p)} \leq \biLip C_\nabla.
\end{align}
 To any function $\psi:Y\to\Rsp$ we associate the vector
\begin{align}\label{eq:lambda_psi}
\bm{\lambda}_\psi :&= (\psi(y_1) - \psi(y_0), \hdots, \psi(y_N) -
  \psi(y_0)) \in \Rsp^N.
  \end{align}
We also consider the same family of convex set as in
Section~\ref{section:localregularity}:
\begin{align*}
  K(\bm{\lambda}) =
  \{ p \in \Xloc \mid \forall 1\leq i\leq N, f_i(p) \leq \lambda_{i} \},
\end{align*}
so that  $K(\bm{\lambda}_\psi) = (\exp_{y_0}^c)^{-1}(\Lag_{y_0}(\psi))$.

\begin{proposition}\label{prop:explicitupperbound1}
  Assume that $\bm{\lambda} := \bm{\lambda}_\psi$ where $\psi$ belongs
  to $\mathcal{K}^{\eps/2}$ and let $p$ be a point in $K(\bm{\lambda})$. Then,
  \begin{description}
  \item[\textbf{Case I}] If $f_i(p) = \lambda_i$ and $f_j(p) =
    \lambda_j$ for  $i\neq j$ in $\{1,\hdots,N\}$, then
    \begin{equation}\label{eq:awayfromcolinear1a}
      \left(\frac{\sca{\nabla f_i(p)}{\nabla f_j(p)}}
           {\nr{\nabla f_i(p)}\nr{\nabla f_j(p)}}\right)^2\leq 1- \delta_1^2.
    \end{equation}
  \item[\textbf{Case II}] If $p \in \partial \Xloc$ and if
    $f_i(p) = \lambda_i$ for some $i$ in $\{1,\hdots,N\}$, then
  \begin{equation}\label{eq:awayfromcolinear1b}
    \forall w\in \Normal_p \Xloc, ~~\left(\frac{\sca{\nabla f_{i}(p)}{w}}
    {\nr{\nabla f_{i}(p)}\nr{w}}\right)^2\leq 1- \delta_1^2.
  \end{equation}
  \end{description}
  In the above inequalities,
   $$\delta_1 :=  \frac{\eps_{nd}\delta_0}{2 \biLip C_\nabla\Ccond^2}.$$
\end{proposition}
By assumption, all the Laguerre cells associated to $\psi$ contain a
mass of at least $\eps/2$.  This allows us to apply
Proposition~\ref{prop:explicitlowerbound}, ensuring that normal
vectors cannot be near-opposite, to all the Laguerre cells in their
exponential charts. We denote $L_i := L_{y_i}(\psi) =
(\exp_{y_i}^c)^{-1}(\Lag_{y_i}(\psi))$ for brevity.

The proposition also relies on two simple lemmas. The first lemma
shows the effect of a diffeomorphism on the normal cone to a convex
set, when its image is also convex.

\begin{lemma}\label{lem:imageofconvexsetnormal}
Let $K\subset \Rsp^d$ be a compact, convex set, $F$ be a $\Class^1$
diffeomorphism from an open neighborhood of $K$ to an open subset of
$\Rsp^d$, and assume that $F(K)$ is also a convex set. Then, for
any point $x$ in $\partial K$ one has
$$ \Normal_{F(x)}(F(K)) = [\D F^{-1}_{F(x)}]^*(\Normal_x
K),$$ where $A^*$ denotes the adjoint of $A$.
\end{lemma}

\begin{proof}
  Consider $x$ in $\partial K$, $v \in \Normal_x K$, and define
  $\phi(z):=\sca{F^{-1}(z)-x}{v}$. Since $v$ is an outer normal to $K$
  at $x$, the restriction of $\phi$ to the set $F(K)$ is
  non-positive. Since $F(K)$ is convex, for any point $y\in K$, $F(K)$
  contains the segment $[F(x),F(y)]$. We therefore have
  \begin{align*}
 0&\geq \phi((1-t)F(x)+tF(y)) \\&\geq \phi(F(x))+t\sca{\nabla
   \phi(F(x))}{F(y)-F(x)}-o(t)\\ &=t\sca{[\D F^{-1}_{F(x)}]^*(v)}{F(y)-F(x)}-o(t).
  \end{align*}
  where we have used $\phi(F(x))=0$ and $\nabla
  \phi(F(x))=[\D F^{-1}_{F(x)}]^*(v)$ to obtain the equality at the end. Dividing by $t$ and
  taking the limit as $t$ goes to zero, we see that
  $$\forall y\in K,~~
  \sca{[\D F^{-1}_{F(x)}]^*(v)}{F(y)-F(x)} \leq 0, $$ thus
  showing that $[\D F^{-1}_{F(x)}]^*(v)$ belongs to the normal
  cone to $F(K)$ at $F(x)$. The converse inclusion follows from
  the symmetry of the problem.
\end{proof}

The second lemma compares the angle between two vectors and the angle
between their image under a linear map, using the generalized
Wiedlandt inequality (see \cite[Section 3.4]{householder1975}).  We
identify $\Rsp^d$ with its tangent and cotangent spaces through the
Euclidean structure. We denote the adjoint of the derivative of the
exponential map $\exp^c_y$ at a point $p$ in $X_y$ by
\begin{align*}
(\D \exp^c_{y_i})^*\vert_{p}:  \Tang^*_{\exp^c_{y_i}(p)}\Omega\to \Tang^*_{p}\Rsp^d\cong\Rsp^d,
\end{align*}

\begin{lemma}\label{lemma:wiedlandt}
  Let $y_k \neq y_\ell \in Y$, let $x$ be a point in $X$ and set $p_{k}
  := (\exp_{y_k}^c)^{-1}(x)$ and $p_\ell := (\exp_{y_\ell}^c)^{-1}(x)$ and
  $$A =
  (\D \exp^c_{y_k}\vert_{p_k})^*\circ [(\D 
    \exp^c_{y_l}\vert_{p_l})^*]^{-1}: \Tang_{p_\ell}^* \Rsp^d \to \Tang_{p_k}^* \Rsp^d.$$
  Then, the following inequalities hold for all $v$, $w$ in $\Rsp^d$:
   \begin{align*}
     \Ccond^{-4}\left(1+\frac{\sca{v}{w}}{\nr{v}\nr{w}}\right) \leq 1 + \frac{\sca{A v}{Aw}}{\nr{Av}\nr{Aw}}
     \leq \Ccond^4\left(1+\frac{\sca{v}{w}}{\nr{v}\nr{w}}\right).
   \end{align*}
   \end{lemma}

\begin{proof}
Indeed, let $\theta$ be the angle between $v$ and $w$ and $\theta'$ be
the angle between $Av$ and $Aw$, both in the interval $(0,\pi)$. Let
$t := \tan(\theta/2)$, $t' := \tan(\theta'/2)$. The generalized Wiedlandt
inequality in \cite[Section 3.4]{householder1975} asserts $(1/\cond(A)) t \leq t' \leq \cond(A) t$. 
Expressing $\cos(\theta)$ in term of $t =
\tan(\theta/2)$,
 \begin{align*}
   1 + \cos(\theta') = 1 + \frac{1-t'^2}{1+t'^2} = \frac{2}{1+t'^2} \leq \cond(A)^2(1+\cos(\theta))
 \end{align*}
 We conclude the second inequality above by using $\cond(A_2^* [A_1^*]^{-1}) \leq
 \cond(A_1)\cond(A_2)$ and the definition of the constant
 $\Ccond$. For the first inequality, simply note that $\cond(A^{-1})=\cond(A)$.
 \end{proof}

\begin{proof}[Proof of Proposition~\ref{prop:explicitupperbound1}, case I]
   We let 
   \begin{align*}
   V :&= \nabla f_i(p) = \nabla f_{y_i,y_0}(p),\\
    W :&=\nabla f_j(p) = \nabla f_{y_j,y_0}(p),\\
     v :&= \frac{V}{\nr{V}},\quad w :=\frac{W}{\nr{W}}.
     \end{align*}
      Switching the indices $i$ and $j$ if necessary, we
   assume that $\nr{V} \leq \nr{W}$. The proof depends on the sign of
   $\sca{W-V}{V}$ (see Remark \ref{rmk:upperestremark} below for the significance of the sign of this expression).  Assume first $\sca{W - V}{V} \leq 0$, and
   let $\alpha_v := 1/\nr{V}, \alpha_w := 1/\nr{W}$. Then,
  \begin{align*}
    1-\sca{v}{w} &=  \frac{1}{2}\nr{v - w}^2 
    = \frac{1}{2}\nr{\alpha_w(W - V) -
      (\alpha_v -\alpha_w) V}^2 \\
    &= \frac{1}{2}\alpha_w^2 \nr{W - V}^2 +
    \frac{1}{2}(\alpha_v - \alpha_w)^2 \nr{V}^2-\alpha_w(\alpha_v - \alpha_w) \sca{W - V}{V}
  \end{align*}
  Using $\alpha_w \leq \alpha_v$, and $\nr{W-V} \geq \eps_{nd}$ we
  end up with
  \begin{align*}
    1-\sca{v}{w}^2 &\geq 1-\sca{v}{w} \geq \frac{1}{2}\alpha_w^2 \nr{W
      - V}^2 \\
    &\geq \frac{1}{2} \frac{\eps_{nd}^2}{\biLip^2 C_\nabla^2}
    \geq \frac{\eps_{nd}^2\delta_0^2}{4 \biLip^2 C_\nabla^2\Ccond^4}= \delta_1^2
  \end{align*} where we have used \eqref{eq:gradlip} and
  \eqref{eq:gradlength}, $\delta_0\leq 1$ and $\Ccond\geq 1$. This
  establishes the desired bound when $\sca{v}{w} \in [0,1]$.  In the
  case $\sca{v}{w} \in [-1,0]$, we can apply
  Proposition~\ref{prop:explicitlowerbound} to show that
  $1-\sca{v}{w}^2 \geq 1+\sca{v}{w} \geq \delta_0^2 \geq \delta_1^2,$
  thus establishing the desired bound.

 Now suppose $\sca{W - V}{V} \geq 0$. A slightly tedious computation
 gives
\begin{align}
  \sca{v}{w}^2&=1 - \frac{\nr{W - V}^2}{\nr{W}^2}+\frac{\sca{W - V}{v}^2}{\nr{W }^2} \notag\\
  &= 1 - \frac{\nr{W - V}^2}{\nr{W}^2} \left(1 - \frac{\sca{W - V}{v}^2}{\nr{W-V}^2}\right) \notag \\
  &\leq 1 - \frac{\eps_{nd}^2}{\biLip^2 C_\nabla^2}\left(1 - \sca{\frac{W - V}{\nr{W-V}}}{v}\right), \label{eq:ub1bvw}
\end{align}
  where we have used $\sca{W - V}{V}\geq 0$ with  \eqref{eq:gradlip} and \eqref{eq:gradlength} to get the last  inequality. We will now apply Proposition
  \ref{prop:explicitlowerbound} to the Laguerre cell $L_i$. By
  Lemma~\ref{lemma:changeoflag}, the point $p_i :=
  (\exp_{y_i}^c)^{-1}(\exp_{y_0}^c (p)) \in X_{y_i}$ belong to
  $L_i$ and the vectors $V_i := \nabla f_{y_0,y_i}(p_i)$ and $W_i:=
  \nabla f_{y_j,y_i}(p_i)$ are both normals to $L_i$ at
  $p_i$. Proposition \ref{prop:explicitlowerbound} then shows that the
  vectors $V_i$ and $W_i$ satisfy 
\begin{equation}
  -1 + \delta_0^2 \leq \frac{\sca{V_i}{W_i}} {\nr{V_i}\nr{W_i}}.\label{eq:LBinLagj}
  \end{equation}
We transfer this inequality to the exponential chart of the original
point $y_0$ using the linear map $$A := (\restr{\D 
  \exp_{y_0}^c}_{p})^* \circ [(\restr{\D 
    \exp_{y_i}^c}_{p_i})^*]^{-1}.$$ First, note that $W - V = AW_i$ and
$V = -A V_i.$ Applying the generalized Wiedlandt inequality
(Lemma~\ref{lemma:wiedlandt}) and \eqref{eq:LBinLagj} we have
\begin{align}
  1 - \frac{\sca{W-V}{v}}{\nr{W-V}} =
  1 + \frac{\sca{A W_i}{A V_i}}{\nr{A W_i}{\nr{A V_i}}} 
  &\geq \Ccond^{-4}
  \left(1 + \frac{\sca{V_i}{W_i}}
  {\nr{V_i}\nr{W_i}}\right) \notag \\
  &\geq \Ccond^{-4} \delta_0^2 \geq \delta_1^2.\label{eq:boundangle2}
\end{align}
Combining this inequality with \eqref{eq:ub1bvw} we obtain
\eqref{eq:awayfromcolinear1a} in this case as well. 
\end{proof}

\begin{proof}[Proof of Proposition~\ref{prop:explicitupperbound1}, case II]
  Consider $V := \nabla f_{i}(p)$ and let $W$ be any vector in the
  normal cone $\Normal_p \Xloc$. When $\sca{V}{W} \leq 0$, the
  inequality directly follows from
  Proposition~\ref{prop:explicitlowerbound}, ensuring that normal
  vectors cannot be near-opposite.  We now assume $\sca{V}{W}\geq 0$
  and we will apply Proposition \ref{prop:explicitlowerbound} to the
  Laguerre cell of $y_i$ and transfer the result to the exponential
  chart of the point $y_0$. Let $p_i =
  (\exp_{y_i}^c)^{-1}(\exp_y^c(p))$. Then, by
  Lemma~\ref{lemma:changeoflag}, $p_i$ belongs to $L_i$ and $V_i: =
  \nabla f_{y_0,y_i}(p_i)$ is a normal vector to $L_i$ at $p_i$. We
  define a second normal vector by considering
  $$A := (\restr{\D  \exp_{y_0}^c}_{p_0})^* \circ [(\restr{\D 
      \exp_{y_i}^c}_{p_i})^*]^{-1}$$ and by setting $W_i := A^{-1}W \in \Tang^*_{p_i}\Rsp^d$.  By
  Lemma~\ref{lem:imageofconvexsetnormal}, the vector $W_i$ belongs to
  the normal cone to $X_{y_i}$ at $p_i$. Moreover, since $L_i$ is
  contained in $X_{y_i}$ and both sets contain $p_i$, we have $\Normal_{p_i}
  X_{y_i} \subseteq \Normal_{p_i} L_i$, thus ensuring that $W_i$ also belongs
  to the normal cone to $L_i$ at $p_i$. Then, by
  Proposition~\ref{prop:explicitlowerbound} again,
$$\frac{\sca{V_i}{W_i}}{\nr{V_i}\nr{W_i}} \geq -1 + \delta_0^2. $$ As
  before, we transfer this inequality to the exponential chart of
  the original point $y$ using the linear map $A$.  We have $V =
  \nabla f_i(p) = -A V_i$, and by construction $W = A W_i$. We get the
  desired inequality by applying Lemma~\ref{lemma:wiedlandt}:
\begin{align*}
  1 - \frac{\sca{V}{W}}{\nr{V}\nr{W}} = 1 + \frac{\sca{AV_i}{AW_i}}{\nr{AV_i}\nr{AW_i}} &\geq
  \Ccond^{-4}\left(1+\frac{\sca{V_i}{W_i}}{\nr{V_i}\nr{W_i}}\right) \\
  &\geq \Ccond^{-4}\delta_0^2 \geq \delta_1^2,
\end{align*}
and by recalling that $\sca{V}{W} \geq 0$.
\end{proof}

\subsection{Proof of Theorem~\ref{th:StrongRegularity}}
\label{subsec:ProofRegularity}
By Theorem~\ref{th:Aurenhammer}, the second-differentiability of
Kantorovich's functional $\Phi$ will follow from the differentiability
of the function
 \begin{align*}
   G_{y_0}(\psi) := \int_{\Lag_{y_0}(\psi)} \rho(x) \dd \Haus^d_g(x)  =
   \int_{L_y(\psi)} \rholoc(p) \dd p 
 \end{align*}
where we have used the change-of-variable formula with $x =
\exp_{y_0}^c(p)$, so that $\rholoc$ is the density of the pushforward
measure $(\exp^c_{y_0})^{-1}_\#(\rho \Haus^d_g)$ with respect to the
Lebesgue measure. We recall that 
\begin{align*}
  K(\bm{\lambda}_\psi) 
  = (\exp_{y_0}^c)^{-1}(\Lag_{y_0}(\psi)),
\end{align*}
so that $G_{y_0}(\psi) = \Gloc(\bm{\lambda}_\psi)$ (as defined in \eqref{eq:G}).   
The differentiability
of $\Gloc$ will be proven using Theorem~\ref{th:LocReg} from the
previous section.

Let us fix a function $\psi_0$ in $\mathcal{K}^\eps$ and recall that
$\bm{\lambda}_0 := \bm{\lambda}_{\psi_0}$.  By
Proposition~\ref{prop:Gcont} there exists a positive constant $T_{tr}$
such that every function $\psi$ on $Y$ satisfying $\nr{\psi -
  \psi_0}_{\infty}\leq T_{tr}$ belongs to
$\mathcal{K}^{\eps/2}$. Then, by
Proposition~\ref{prop:explicitupperbound1}, we see that the functions
$f_i$ satisfy the transversality condition \eqref{eq:T} on the cube
$\bm{\lambda_0} + [-T_{tr}, T_{tr}]^N$ with constant
\begin{equation}
  \eps_{tr} = \delta_1 = 
  \frac{\eps_{nd}\delta_0}{2 \biLip C_\nabla\Ccond^2},
  \label{eq:EpsTRConstants}
\end{equation}
where we recall that $\delta_0 =
\eps/(2^d\nr{\rho}_\infty{\biLip}^{2d}\diam{(X)}^{d})$.  Note also
that since $\rho$ is $\alpha$-H\"older and since the exponential map
is $\Class^{1,1}$, the probability density $\rholoc$ is also
$\alpha$-H\"older with constant
\begin{equation}
  \label{eq:Rholoc}
  \nr{\rholoc}_{\Class^{0,\alpha}(\Xloc)} \leq
  C(\nr{\rho}_{\Class^{0,\alpha}},\Cdet).
\end{equation}
We can now apply Theorem \ref{th:LocReg}. This ensures that the
function $\Gloc$ is of class  $\Class^{1, \alpha}$ on the cube
$\bm{\lambda_0} + [-T_{tr}, T_{tr}]^N$, so that
$\partial\Phi/\partial\one_{y_0}$ is $\Class^{1,\alpha}$ on a
neighborhood of $\psi_0$. Since this holds for any point $y_0\in Y$ and
any function $\psi_0$ in $\mathcal{K}^\eps$, we have established the
$\Class^{2,\alpha}$-regularity of $\Phi$ on $\mathcal{K}^\eps$. The
claimed dependency of $\nr{\Phi}_{\Class^{2,
    \alpha}(\mathcal{K}^\eps)}$ follows from equations
\eqref{eq:EpsTRConstants}--\eqref{eq:Rholoc} and from
Theorem~\ref{th:LocReg}.

Our goal is now to deduce the formula for the gradient of $G$ given in
Theorem~\ref{th:StrongRegularity} (Equation~\eqref{eq:Hess}), from the
formula for the gradient of $\Gloc$ given in Theorem~\ref{th:LocReg}
(Equation~\eqref{eq:PD}). This is done by looking more closely at the
change of variable induced by the exponential map $F := \exp^c_{y_0}:
\Omega\to \Rsp^d$. For ease of notation we let $h:= c(\cdot, y_0) - c(\cdot, y_i)\lambda
= f_i\circ F^{-1}$. By definition of the push-forward, we have for any
bounded measurable function $\chi$ on $\Omega$,
$$ \int_{\hat{\Omega}} \chi(F(p)) \rho(p) \dd\Haus^{d}(p) =
\int_{\Omega} \chi(x) \hat{\rho}(x) \dd \Haus^{d}_{g}(x). $$
Multiplying $\chi$ by the characteristic function of $h^{-1}([t,s])$, this gives
$$ \int_{f_i^{-1}([t, s])} \chi(F(p)) \hat{\rho}(p) \dd\Haus^{d}(p) =
\int_{h^{-1}([t,s])} \chi(x) \rho(x) \dd \Haus^{d}_{g}(x).$$
Applying the coarea formula on both sides, we get
\begin{equation} \int_{t}^{s} \int_{f_i^{-1}(r)} \frac{\chi(F(p))
  \hat{\rho}(p)}{\nr{\nabla f_i(p)}} \dd\Haus^{d-1}(p)\dd r =
  \int_{t}^{s} \int_{h^{-1}(r)} \frac{\chi(x) \rho(x)}{\nr{\nabla
      h(x)}_{g}} \dd \Haus^{d-1}_{g}(x) \dd r. \label{eq:coarea}
\end{equation}
Using the
$\Class^{1,1}$ smoothness of the functions $f_i$ and the \eqref{eq:Tw}
condition, we can see that for any $\chi$ in $\Class^0_c(\Omega)$, the
two inner integrals
$$ r \mapsto \int_{f_i^{-1}(r)} \frac{\chi(F(p))
  \hat{\rho}(p)}{\nr{\nabla f_i(p)}} \dd\Haus^{d-1}(p) \hbox{ and }
r\mapsto \int_{h^{-1}(r)} \frac{\chi(x) \rho(x)}{\nr{\nabla
    h(x)}_{g}} \dd \Haus^{d-1}_{g}(x)
$$ depend continuously on $r$. Using the continuity of these functions
in $r$, equation \eqref{eq:coarea} and the Fundamental Theorem of
Calculus, we get that for any function $\chi$ in $\Class^0_c(\Omega)$
and any $r$ in $\Rsp$,
$$\int_{f_i^{-1}(r)} \frac{\chi(F(p)) \hat{\rho}(p)}{\nr{\nabla
    f_i(p)}} \dd\Haus^{d-1}(p) = \int_{h^{-1}(r)} \frac{\chi(x)
  \rho(x)}{\nr{\nabla h(x)}_{g}} \dd \Haus^{d-1}_{g}(x).$$ By
Tietze's extension theorem, \eqref{eq:Tw} and \eqref{eq:Reg}, the level set $S := h^{-1}(r)$ is a  $\Class^{1,1}$ hypersurface of $\Omega$. Thus
every function in $\Class^0_c(S)$ can be extended to a function in
$\Class^0_c(\Omega)$. The previous equality therefore holds for any
$\chi$ in $\Class^0_c(S)$, and by density, it also holds for any
function $\chi$ in $\LL^1(S)$. Applying this with $\chi$ equal to the
indicator function of the interface between the Laguerre cell of $y_0$
and the cell of $y_i$, we get the desired formula for the partial
derivatives:
\begin{align*}
  \frac{\partial G_i}{\partial \psi_i}(\psi) =
  \frac{\partial \Gloc}{\partial \lambda_i}(\bm{\lambda}_\psi) 
  &=
  \int_{L_{y_0}(\psi) \cap f_i^{-1}(\psi(y_i) - \psi(y_0))} 
  \frac{\rholoc(p)}{\nr{\nabla f_i(p)}} \dd \Haus^{d-1}(p) \\
  &= \int_{\Lag_{y_0}(\psi) \cap \Lag_{y_i}(\psi)} 
  \frac{\rho(x)}{\nr{D c_{y_0}(x) - D c_{y_i}(x)}_{g}} \dd \Haus_g^{d-1}(x).
\end{align*}

\subsection{Alternative upper transversality estimates}\label{subsec:alternative-bound}
\label{subsec:regularitymtw}
Finally, we state an alternate upper transversality estimate, under the assumption that the points in $Y$ are sampled from some \emph{target domain} $\Lambda$, along with some convexity conditions. Specifically, let $\Lambda$ be a bounded, open subset in some Riemannian manifold, with $Y\subset \Lambda$. We then assume that for any $x'\in \Omega^{\cl}$, 
the mapping
\begin{align*}
 y&\mapsto -D_xc(x', y)
\end{align*}
is a diffeomorphism onto its range, and we denote the inverse by $\exp^c_{x'}$. We will also assume that 
%
%
 $\coord{\Lambda}{x}$ is convex for all $x\in \Omega$, and finally that the following inequality holds: for any $x$, $x'\in \Omega$, $q_0$, $q_1\in \coord{\Lambda}{x'}$, and $t\in [0, 1]$,
\begin{align}
&-c(x, \exp_{x'}^c((1-t)q_0+tq_1))+c(x',  \exp_{x'}^c((1-t)q_0+tq_1)))\notag\\
&\leq \max\{-c(x, \exp_{x'}^c(q_0))+c(x',  \exp_{x'}^c(q_0)), -c(x, \exp_{x'}^c(q_1))+c(x',  \exp_{x'}^c(q_1))\}\label{eq:targetloeper}
\end{align}
Note that this last inequality is nothing but quasi-concavity of $c(x',\cdot) - c(x,\cdot)$ in the global coordinate chart of $\Lambda$ defined by $\exp_{x'}^c$.
For more on these conditions, see Remark \ref{rem:MTWremark} below.

Proposition \ref{prop:explicitupperbound2} can be applied to provide an alternative bound in the transversality condition \eqref{eq:T} when the point $p_0\in \partial K(\bm{\lambda})$ is in the interior of $X$ (so in particular, when dealing with Laguerre cells that do not intersect $\partial X$). The advantage of this bound is that it does not require knowledge of the condition number $\Ccond$.

Recall that we have fixed some point $y_0\in Y=\{y_1,\ldots, y_N\}$ and for any index $i\in \{1,\ldots, N\}$ we use the notation,
\begin{align*}
 f_i(p)=f_{y_i, y_0}(p)=c(\exp^c_{y_0}{p}, y_0)-c(\exp^c_{y_0}{p}, y_i).
\end{align*}
We also re-define the constants $C_\nabla$ and $\biLip$ so that in their definitions, the maximum of $y$ ranges over the domain $\Lambda$ instead of just $Y$.
\begin{proposition}\label{prop:explicitupperbound2}
Suppose $\nr{\bm{\lambda}}<T_{tr}<\frac{\eps}{8\biLip^{2d-2}\nr{\rho}_\infty\Haus^{d-1}_g(\partial X)}$, and $p_0\in K(\bm{\lambda})$ with $f_i(p_0) = \lambda_i$ and $f_j(p_0) =
    \lambda_j$ for  $i\neq j$ in $\{1,\hdots,N\}$. Then we have
    \begin{equation}\label{eq:awayfromcolinear2a}
      \left(\frac{\sca{\nabla f_i(p_0)}{\nabla f_j(p_0)}}
           {\nr{\nabla f_i(p_0)}\nr{\nabla f_j(p_0)}}\right)^2\leq 1- \delta_2^2
    \end{equation}
    where
 \begin{align*}
\delta_2:=\frac{\eps\eps_{nd}}{4\sqrt{2}C_\nabla^2\biLip^{2d+4}\nr{\rho}_\infty\left(\Haus^{d-1}_g(\partial X)\right)}.
\end{align*}

\end{proposition}

\begin{remark} \label{rmk:upperestremark}
Before embarking on the proof of this ``continuous'' upper transversality estimate, we compare some key features of its proof with that of the ``discrete'' upper transversality estimate Proposition \ref{prop:explicitupperbound1}. By considering the case when the two vectors $\nabla f_i(p_0)$ and $\nabla f_j(p_0)$ are collinear, we can see that both proofs rely on the same core idea. In this case, $\nabla f_i(p_0)$ and $\nabla f_j(p_0)$ are outward normal vectors (in coordinates induced by $\exp^{c}_{y_0}$, see Remark \ref{rem:MTWremark}) to the sublevel sets $\{-c(\cdot, y_0)+\psi(y_0)\leq -c(\cdot, y_i)+\psi(y_i)\}$ and $\{-c(\cdot, y_0)+\psi(y_0)\leq -c(\cdot, y_j)+\psi(y_j)\}$ respectively, at $p_0$ which lies on the intersection of their boundaries. Since these sets are convex in the associated coordinates, this will cause the Laguerre cell associated to $y_i$ to be trapped in a lower dimensional set, giving it zero mass which is a contradiction. The difference between the two proofs lies in quantifying this estimate. In the discrete version of the estimate we do one of two things depending on the sign of the inner product $\sca{W - V}{V}$ (see proof of Proposition \ref{prop:explicitupperbound1}). When this inner product is negative, using that $\nr{W}\geq \nr{V}$ and the lower bound on $\nr{W-V}$ from the non-degeneracy condition \eqref{eq:ND} yields that $W$ must lie outside of a cone of a certain size opening with axis along $V$. In the other case we note $W-V$ and $-V$ are, respectively, outward normal vectors to the sublevel sets $\{-c(\cdot, y_i)+\psi(y_i)\leq -c(\cdot, y_j)+\psi(y_j)\}$ and $\{-c(\cdot, y_i)+\psi(y_i)\leq -c(\cdot, y_0)+\psi(y_0)\}$, viewed in coordinates given by $\exp_{y_i}^{c}$. Thus the lower transversality estimate Proposition \ref{prop:explicitlowerbound} can be applied to obtain a quantitative bound, but at the price of involving the condition number since we have made a change of coordinates. In the continuous version, there is no change of coordinates, instead we make a rotation to align $\nabla f_j(p_0)$ and $\nabla f_i(p_0)$, then estimate the error induced by this rotation using \eqref{eq:targetloeper}, in a vein similar to calculations from \cite[Remark 2.5, Proof of Lemma 4.7]{guillenkitagawa2015}.
\end{remark}
\begin{proof}[Proof of Proposition \ref{prop:explicitupperbound2}]
Let us again write 
 \begin{align*}
   V :&= \nabla f_i(p_0),\quad
    W :=\nabla f_j(p_0),\\
     v :&= \frac{V}{\nr{V}},\quad w :=\frac{W}{\nr{W}},
     \end{align*}
and assume $\eps_{nd}<\nr{V}\leq \nr{W}$ and $\sca{v}{w}>0$. Let us also define
\begin{align*}
x_0:&=\exp^{c}_{y_0}(p_0),\quad
q_0:=-D_xc(x_0, y_0),\quad q_1:=-D_xc(x_0, y_j).
\end{align*}
A quick calculation yields
\begin{align*}
 q_0&=[(\D  \exp^{c}_{y_0}\vert_{p_0})^*]^{-1}(-\nabla_p c(\exp^{c}_{y_0}(p), y_0)\vert_{p=p_0}),\\
 q_1&=[(\D  \exp^{c}_{y_0}\vert_{p_0})^*]^{-1}(W)+q_0,
\end{align*}
Now we define the point
\begin{align*}
 q':=[(\D  \exp^{c}_{y_0}\vert_{p_0})^*]^{-1}(\nr{V}w)+q_0,
\end{align*}
since $\nr{V}\leq \nr{W}$, the above calculation yields that $q'$ lies on the line segment between $q_0$ and $q_1$; since $\coord{\Lambda}{x_0}$ is convex we have that $q'\in \coord{\Lambda}{x_0}$ as well.

%
%
Thus we can define
\begin{align*}
 y'_i:&=\exp^c_{x_0}(q'),\\
  \tilde{f}_i(p):&=-c(\exp^{c}_{y_0}(p), y'_i)+c(\exp^{c}_{y_0}(p), y_0)+c(x_0, y'_i)-c(x_0, y_0)+\lambda_i,
\end{align*}
and by \eqref{eq:targetloeper} applied with the choices $x:=\exp^{c}_{y_0}(p)$, $x':=x_0$, $t:=\frac{\nr{V}}{ \nr{W}}$, and $q_0$, $q_1$ as defined above.
 we will obtain for all $p\in \coord{\Omega}{y_0}$,
 \begin{align}\label{eq:DASM}
 \tilde{f}_i(p)-\lambda_i\leq \max\{0, f_j(p)-\lambda_j\},
\end{align}
while another quick calculation yields
\begin{align*}
 y_i&=\exp^c_{x_0}([(\restr{\D  \exp^{c}_{y_0}}_{p_0})^*]^{-1}(V)+q_0).
\end{align*}
Now note that 
\begin{align*}
& \lvert -c(\exp^{c}_{y_0}(p), y'_i)+c(\exp^c_{y_0}(p), y_i)\rvert\\
&\leq \sup_{(x, q)\in \Omega\times \coord{\Lambda}{x_0}}\nr{(\restr{\D \exp^c_{x_0}}_{q})^* (-D_yc(x, \exp^c_{x_0}(q)))}\nr{[(\restr{\D  \exp^{c}_{y_0}}_{p_0})^*]^{-1}(\nr{V}w-V)}\\
&\leq C_\nabla \biLip^2\nr{\nr{V}w-V},
\end{align*}
where we have used that if $y=\exp^c_{x_0}(q)$, then $(\restr{\D \exp^c_{x_0}}_{q})^*=\restr{\D \exp^{c}_{y}}_{-D_x c(x_0, y)}$. As a result we obtain
\begin{align*}
 &\lvert \tilde{f}_i(p)-f_i(p)\rvert^2=\lvert -c(\exp^{c}_{y_0}(p), y'_i)+c(\exp^{c}_{y_0}(p), y_i)\rvert^2\\
 &\leq (C_\nabla \biLip^2)^2\nr{V-\nr{V}w}^2
 =2(C_\nabla \biLip^2)^2\nr{V}^2(1-\sca{v}{w})\\
 &\leq 2(C_\nabla \biLip^2)^2(\biLip C_\nabla)^2(1-\sca{v}{w}).
\end{align*}
Combining with \eqref{eq:DASM} we then have for any $p\in \coord{\Omega}{y_0}$,
\begin{align*}
 f_{i}(p)-\lambda_i\leq \max\{0,\ f_j(p)-\lambda_j\}+\sqrt{2} \biLip^3 C_\nabla^2 \sqrt{1-\sca{v}{w}}
\end{align*}
or re-arranging and using that $\nr{\bm{\lambda}}<T_{tr}$,
\begin{align}
 1-\sca{v}{w} &\geq \sup_{p\in \coord{\Omega}{y_0}}\frac{(f_i(p)-\max\{0, f_j(p)\}-2T_{tr})^2}{2(\biLip^3 C_\nabla^2)^2}.\label{eq:Loeperanglebound1}
\end{align}

We now make the following observation. Let us write $X_i:=\coord{X}{y_i}$. Then for any $t$, $s>0$, we can estimate the volume of $X_i\cap\{f_{y_0, y_i}\leq -t\}\cap \{f_{y_j, y_i}\leq -s\}$ by
\begin{align*}
& \Haus^d\left(X_i\cap\{f_{y_0, y_i}\leq -t\}\cap \{f_{y_j, y_i}\leq -s\}\right)\\
 &\geq \Haus^d\left(X_i\cap\{f_{y_0, y_i}\leq 0\}\cap \{f_{y_j, y_i}\leq 0\}\right)-\Haus^d\left(X_i\cap\{-t< f_{y_0, y_i}\leq 0\}\right)\\
 &\quad -\Haus^d\left(X_i\cap \{-s< f_{y_j, y_i}\leq 0\}\right).
\end{align*}
Using that $L_i\subset \{f_{y_0, y_i}\leq 0\}\cap \{f_{y_j, y_i}\leq 0\}$, we can bound the first term from below as
\begin{align*}
 \Haus^d\left(X_i\cap\{f_{y_0, y_i}\leq 0\}\cap \{f_{y_j, y_i}\leq 0\}\right)\geq \frac{\eps}{\biLip^d\nr{\rho}_\infty}.
\end{align*}
For the second term, by the coarea formula, we can write
\begin{align*}
 \Haus^d\left(X_i\cap\{-t< f_{y_0, y_i}\leq 0\}\right)&\leq \int_{-t}^0 \int_{X_i\cap \{f_{y_0, y_i}=z\}}\frac{1}{\nr{\nabla f_{y_0, y_i}(p)}}d\Haus^{d-1}(p)dz\\
 &\leq \frac{t \Haus^{d-1}(\partial X_i)}{\eps_{nd}}\leq \frac{t\biLip^{d-1}\Haus^{d-1}_g(\partial X)}{\eps_{nd}}, 
\end{align*}
where to obtain the second line we have again used the fact that for every $z\in \Rsp$, the set $X_i\cap \{f_{y_0, y_i}=z\}$ is contained in the boundary of a convex subset of $X_i$ in conjunction with \cite[Remark 5.2]{kitagawa2014iterative}. By a similar bound on the third term, we see that as long as
\begin{align*}
 \max\{t,\ s\}<\frac{\eps\eps_{nd}}{2\biLip^{2d-2}\nr{\rho}_\infty\Haus^{d-1}_g(\partial X)}
 \end{align*}
 we have
 \begin{align*}
 \Haus^d\left(X_i\cap\{f_{y_0, y_i}\leq -t\}\cap \{f_{y_j, y_i}\leq -s\}\right)>0,
\end{align*}
thus in particular, (by continuity of $f_{y_0, y_i}$ and $f_{y_j, y_i}$) there must exist a point $p'_c\in X_i$ for which $\max\{f_{y_0, y_i}(p'_c),\ f_{y_j, y_i}(p'_c)\}\leq -\frac{\eps\eps_{nd}}{2\biLip^{2d-2}\nr{\rho}_\infty\Haus^{d-1}_g(\partial X)}$. Translating this back into coordinates in $\coord{X}{0}$ and in terms of $f_i$, $f_j$, we see there exists a point $p_c\in \coord{X}{0}$ for which 
\begin{align*}
f_i(p_c)- \max\{0,\ f_j(p_c)\}\geq \frac{\eps\eps_{nd}}{2\biLip^{2d-2}\nr{\rho}_\infty\Haus^{d-1}_g(\partial X)}.
\end{align*}
 Thus if we have $T_{tr}\leq \frac{\eps\eps_{nd}}{8\biLip^{2d-2}\nr{\rho}_\infty\Haus^{d-1}_g(\partial X)}$, combining with \eqref{eq:Loeperanglebound1} we will obtain the bound \eqref{eq:awayfromcolinear2a} as desired.
\end{proof}

\begin{remark}\label{rem:MTWremark}
Under a set of standard conditions, we can obtain both \eqref{eq:MTW} and \eqref{eq:targetloeper}.

Let $\Omega$ and $\Lambda$ be bounded and smooth domains in $d$ dimensional Riemannian manifolds and take a cost $c\in C^4(\overline{\Omega} \times \overline{\Lambda})$. Also assume
\begin{itemize}
\item $c$ satisfies the \eqref{eq:Tw} condition: for every $x\in \Omega$, the map $y\in \Lambda\mapsto -D_x c(x,y)$ is a diffeomorphism onto its image $\Lambda_x:=-D_x c (x, \Lambda)\subset T^*_x\Omega$ and we define the $c$-exponential map $\exp^c_x: \Lambda_x \to \Lambda$  by $\exp^c_x = (-D_x c(x, \cdot))^{-1}$. 
\item the cost $c^*(x,y):=c(y,x)$ satisfies the \eqref{eq:Tw} condition: for every $y\in \Lambda$, we can define the $c^*$-exponential map $\exp^{c}_y: \Omega_y \to \Omega$ on the set $\Omega_y:=-D_y c (\Omega, y)\subset T^*_y\Lambda$ by $\exp^{c}_y = (-D_y c(\cdot, y))^{-1}$. 
\item $(\exp^c_x)^{-1}(\Lambda)$ is convex for each $x\in \Omega$.
\item $(\exp^{c^*}_y)^{-1}(\Omega)$ is convex for each $y\in \Lambda$.
\item $\det D^2_{xy} c(x, y) \neq 0$ for all $(x,y)\in \overline{\Omega} \times \overline{\Lambda}$.
\item For any $(x, y)\in \overline{\Omega}\times \overline{\Lambda}$ and $\eta\in T^*_x\Omega$, $V\in T_x\Omega$ with $\eta(V)=0$, 
\begin{align}\label{A3w}\tag{A3w}
-(c_{ij, pq}-c_{ij, r}c^{r, s}c_{s, pq})c^{p, k}c^{q, l}V^iV^j\eta_k\eta_l\geq 0,
\end{align}
\end{itemize}
here indices before a comma are derivatives on $\Omega$ and after a comma on $\Lambda$, for fixed coordinate systems, and a pair of raised indices denotes the inverse of a matrix. This last condition \eqref{A3w} originates (in a stronger version) in \cite{ma2005regularity} related to regularity of optimal transport. \cite[Theorem 3.2]{loeper2009regularity} in the Euclidean case and \cite[Theorem 4.10]{kimmccann10} in the general manifold case show the above conditions imply \eqref{eq:MTW} and \eqref{eq:targetloeper}. In fact, they are equivalent as seen in \cite{loeper2009regularity}. This geometric interpretation is a key ingredient in showing regularity in the optimal transport problem in the vein of Caffarelli's classical work \cite{caffarelli92}, see \cite{figallikimmccann13, guillenkitagawa2015}.
%
%
\end{remark}

\section{Strong concavity of Kantorovich's functional}
\label{sec:Concavity}
We establish in this section the strong concavity of Kantorovich's
functional $\Phi$ over some suitable domain of $\Rsp^Y$. As explained
in the introduction, $\Phi$ is invariant under addition of a constant,
so that we must restrict ourselves to the orthogonal complement $E_Y$
of the space of constant functions. Moreover, we will consider the set
$\K^\eps$ defined by \eqref{eq:Keps}, which can be thought of as the
space of strictly $c$-concave functions. Recall that the conditions \eqref{eq:Reg}, \eqref{eq:Tw}, \eqref{eq:MTW}, and \eqref{eq:L1-poincare} are defined in the introduction, on pages \pageref{eq:Reg}, \pageref{eq:Tw}, \pageref{eq:MTW}, and \pageref{eq:L1-poincare} respectively.

\begin{theorem}\label{th:UniformConcavity}
  Assume \eqref{eq:Reg}, \eqref{eq:Tw}, \eqref{eq:MTW}. Let $X$ be a
  compact, $c$-convex subset of $\Omega$, and $\rho$ be a continuous
  probability density on $X$ satisfying \eqref{eq:L1-poincare}. Then,
  there exists a positive constant $C$, such that
$$
\forall \psi \in \K^\eps,~\forall v\in E_Y,~~ \sca{\D^2\Phi(\psi) v}{v} \leq -C \cdot\varepsilon^3 \nr{v}^2,
$$ 
where $C$ is defined in \eqref{eq:StrongConcavityConstant}, and depends on $\nr{\rho}_{\infty}$, $\Haus^{d-1}_g(\partial X)$, and $\biLip$, $C_\nabla$, and $\eps_{tw}$ from Remark~\ref{rem:Constants}.
\end{theorem}
\begin{remark}\label{rem:nonsharpnessofisoperimetric} Note that the upper bound on the largest non-zero
    eigenvalue of $\D^2\Phi(\psi)$ degrades as $N$ grows to infinity,
    since $\eps$ is of the order of $1/N$. A possible place for
    improvement is the reverse isoperimetric inequality stated in
    equation~\eqref{eq:Hausbd}. Currently, we are vastly overestimating the size of the boundary of a Laguerre cell by bounding it with the area of the boundary of the whole domain, additionally we are bounding the density $\rho$ by its supremum, and paying in terms of the constant $\biLip$. Note that \eqref{eq:Hausbd} in its current form can \emph{never} achieve equality, even for constant density $\rho$ and the quadratic cost function where $\biLip=1$, as equality would only happen for a Laguerre cell that occupied the whole domain $X$, which can not happen as \emph{all} Laguerre cells have nonzero mass. 
     To improve the inequality, one
    could try to control the anisotropy of Laguerre cells and bound the area of the boundary of a cell by some fraction of the area of $\partial X$, however this would require 
    assumptions on the distribution of the points $Y$ and on
    $\nu\in\Prob(Y)$. We believe that such an upper bound on the
    anisotropy of Laguerre cells would be interesting in itself, and
    heuristically seems feasible to view as a discrete analogue of the regularity results
    for optimal transport (interpreting the Laguerre cells associated
    to $\psi:Y\to\Rsp$ as the $c^*$-subdifferentials of $\psi$).
\end{remark}

\begin{remark}
  Note that unlike the domain $X$, the support of the density $\rho$
  does not need be $c$-convex. We provide in
  Appendix~\ref{app:poincare} an example of a radial measure on
  $\Rsp^d$ whose support is an annulus (hence is not simply connected)
  but whose Poincaré-Wirtinger constant \eqref{eq:L1-poincare} is
  nonetheless positive.
\end{remark}

The end of the section is devoted to the proof of Theorem
\ref{th:UniformConcavity}. It relies on the fact that
$-\D^2\Phi(\psi)$ can be regarded as the Laplacian matrix of a
weighted graph on $Y$, whose first nonzero eigenvalue can be
controlled from below using the Cheeger constant of the weighted
graph. In turn, this weighted Cheeger constant can be controlled using
the Poincar\'e-Wirtinger inequality.

\subsection{Poincar\'e inequality and continuous Cheeger constant}
We start by proving that the finiteness of the Poincaré-Wirtinger
constant of the weighted domain $(X,\rho)$ implies the positivity of
the weighted Cheeger constant, defined in \eqref{eq:ContCheeger}. In the following, a \emph{Lipschitz domain} denotes  the closure of an open set with Lipschitz
boundary.

\begin{lemma} Assume \eqref{eq:MTW} and that $X$ is compact and $c$-convex. Then 
  \begin{itemize}
  \item[(i)] $X$ is a Lipschitz domain,
  \item[(ii)] for any $\psi \in \K^+$ and $y$ in $Y$, $\Lag_y(\psi) \cap X$
    is a Lipschitz domain.
  \end{itemize}
  \label{lem:LipDomain}
\end{lemma}
\begin{proof}
  By assumption, for any $y \in Y$ one can write $X = \exp_y^c(X_y)$
  where $X_y$ is a bounded convex subset of $\Rsp^d$ which must have
  nonempty interior since it supports an absolutely continuous
  probability measure. Moreover, the map $\exp_y^c$ is a
  diffeomorphism, hence bi-Lipschitz. This implies (i), while (ii)
  follows from the exact same arguments, remembering that
  $\rho(\Lag_y(\psi)) > 0$.
\end{proof}
Given a Lipschitz domain $A$ of $X$ we write, slightly abusing notations, 
$$\abs{\partial A}_\rho := \int_{\partial A \cap \interior{X}} \rho(x)
\d\Haus^{d-1}_g(x) ~ \hbox{ and }~ \abs{A}_\rho := \int_{A} \rho(x) \d\Haus^{d}_g(x).$$ 
\begin{lemma} \label{lem:PWtoCheeger}
Let $X$ be a compact domain of $\Omega$ and $\rho$ in $\Class^0(X)$ be
a probability density with finite Poincaré-Wirtinger constant
\eqref{eq:L1-poincare}. Then the weighted Cheeger constant of
$(X,\rho)$ is positive, that is
\begin{equation}
\h(\rho) := \inf_{A \subseteq X}
\frac{\abs{\partial A}_\rho}{\min(\abs{A}_\rho, \abs{X\setminus A}_\rho)}
 \geq \frac{2}{\Cpw}, \label{eq:ContCheeger}
\end{equation} where the infimum is taken over Lipschitz domains $A\subseteq
 \interior{X}$ whose boundary has finite $\Haus^{d-1}_g$--measure.
\end{lemma}
The proof is based on properties of functions with bounded
variation. For more details on this topic, we refer the reader to
\cite{attouch2014variational}. Although the discussion in the reference is on Euclidean spaces, the relevant results easily extend to the Riemannian case, as $\exp^c_y$ serves as a global coordinate system on all of $\Omega$.
\begin{proof} Let $A$ be a Lipschitz domain $A$ of $\interior{X}$. 
  Since $A$ has a Lipschitz boundary with finite
  area, its indicator function $\chi_A$ has bounded variation in
  $\interior{X}$. By the density theorem \cite[Theorem
    10.1.2]{attouch2014variational}, there exists a sequence of
  $\Class^1$-functions $f_n$ on $\interior{X}$ that converges to
  $\chi_A$ in the sense of intermediate convergence (whose definition is not important here). By
  \eqref{eq:L1-poincare},
$$
\nr{f_n-\Exp_\rho(f_n)}_{L^1(\rho)} \leq \Cpw\ \nr{\nabla f_n}_{L^1(\rho)}.
$$ Since intermediate convergence is stronger than $L^1$
convergence, the continuity of $\rho$ implies
$$
\lim_{n \to \infty} \nr{f_n-\Exp_\rho(f_n)}_{L^1(\rho)} = \nr{\chi_A-\Exp_\rho(\chi_A)}_{L^1(\rho)}
=  2 \abs{A}_\rho\ \abs{X\setminus A}_\rho.
$$ Note that we used the fact that $\rho$ is a probability measure,
i.e.  $\rho(X) = 1$.  Proposition 10.1.2 of
\cite{attouch2014variational} implies that the total variation measure
$|Df_n|$ narrowly converges to $|D\chi_A|$, which with the continuity of
$\rho$ implies that $\int_{\Omega}  |Df_n|\rho\d \Haus^{d}_g$ converges to
$\int_{\Omega}  |D\chi_A| \rho\d \Haus^{d}_g= \abs{\partial A}_\rho$. The relation
$|Df_n|=\nr{\nabla f_n}_g \d \Haus^{d}_g$ then gives
$$
\lim_{n\to \infty} \nr{\nabla f_n}_{L^1(\rho)} \leq \abs{\partial A}_\rho.
$$ Combining the previous equations together leads to the desired
inequality.
\end{proof}

\subsection{Cheeger constant of a graph}
The goal of this section is to give a lower bound of the second
eigenvalue of $-\D^2\Phi(\psi)$ in terms of the Cheeger constant of
the weighted graph induced by this matrix. An unoriented weighted
graph can always be represented by its adjacency matrix
$(w_{yz})_{(y,z) \in Y^2}$, a symmetric matrix with zero diagonal
entries.  We introduce a few definitions from graph theory, following
the conventions of \cite{friedland2002cheeger}.

\begin{definition} Let $(w_{yz})_{(y,z) \in Y^2}$ be a weighted graph over $Y$.
  The (weighted) \emph{degree} of a vertex $y$ is $d_y := \sum_{z\neq
    y} w_{yz}$. The (weighted) \emph{Laplacian} is the matrix
  $(L_{yz})_{(y,z)\in Y^2}$ whose entries are $L_{yz} = -w_{yz}$ for
  $y\neq z$ and $L_{yy} = d_y$.
\end{definition}
\begin{definition} The \emph{Cheeger constant}
  of a weighted graph $(w_{yz})_{(y,z) \in Y^2}$ over a point set $Y$
  is given by
       $$ \h(w) := \min_{S\subseteq Y} \frac{\abs{\partial
           S}_{w}}{\min(\abs{S}_{w}, \abs{Y\setminus S}_{w})},
       $$
 $$ \hbox{where }\abs{\partial S}_{w} := \sum_{y\in S, z\not\in S} w_{yz} \quad\hbox{
         and } \quad \abs{S}_{w} := \sum_{y\in S} d_y. $$
 \end{definition}

The (weighted) Cheeger inequality bounds from below the first nonzero
eigenvalue of the Laplacian of a weighted graph, denoted $\lambda(w)$,
from its Cheeger constant and its minimal degree. The formulation we
use can be deduced from Corollary 2.2 of \cite{friedland2002cheeger}
and from the inequality $1-\sqrt{1-x^2} \geq x^2/2.$
\begin{theorem}[Cheeger inequality]\label{th:cheeger}
  $
  \lambda(w) \geq \frac{1}{2}\ \h^2(w) \cdot \min_{y\in Y} d_y.
$
\end{theorem}
We now proceed to the proof of the main theorem of this section.

\subsection{Proof of Theorem~\ref{th:UniformConcavity}}
Let $\psi$ be a function in $\K^\eps$ and consider the
weighted graph $(w_{yz})_{(y,z)\in Y^2}$ given by
\begin{equation*}
 w_{yz} := -\frac{\partial^2\Phi}{\partial \one_y\partial \one_z}
 (\psi) = \int_{\Lag_{y,z}(\psi)} \frac{\rho(x)}{\nr{D_x c(x,y) -
     D_x c(x,z)}_{g}}\d\Haus^{d-1}_g(x)
\end{equation*}
for $y\neq z$ in $Y$, and with zero diagonal entries ($w_{yy} =
0$). In the formula above, we used the notation
$\Lag_{y,z}(\psi)=\Lag_{y}(\psi) \cap \Lag_{z}(\psi)$ for the facet
between two Laguerre cells. By construction, the Laplacian matrix of
this weighted graph is the Hessian matrix $-\D^2\Phi(\psi)$, so that
Theorem~\ref{th:cheeger} directly gives us a lower bound on the first
nonzero eigenvalue of $-\D^2\Phi(\psi)$. To complete the proof, we
need to bound the Cheeger constant and the minimum degree of the graph
$w$ from below.

\subsubsection{Step 1} 
The goal here is to bound from below the discrete Cheeger constant
$\h(w)$ in terms of the continuous weighted Cheeger constant
$\h(\rho)$ and the constants introduced in \eqref{eq:constants}. By
definition of the constants $\eps_{tw}$ and $C_\nabla$, for every
$y\neq z$ in $Y$, one has
\begin{equation}
\eps_{tw}\ w_{yz} 
\leq |\Lag_{y,z}(\psi)|_{\rho}
\leq 2 C_\nabla w_{yz}.
\label{eq:cmpWrho}
\end{equation}
Consider a subset $S$ of $Y$, and let $A = \cup_{y\in S}
\Lag_y(\psi)$.  Then, the intersection of the boundary of $A$ with $X$
is contained an union of facets of Laguerre cells, namely
\begin{equation}
  \partial A \cap \interior{X} \subseteq \bigcup_{y \in S,z\not\in S} \Lag_{y,z}(\psi).\label{eq:bdA}
\end{equation}
The two inequalities \eqref{eq:cmpWrho} and \eqref{eq:bdA} imply a
lower bound on the numerator of the Cheeger constant:
\begin{equation}\label{eq:cheeger-numerator}
\abs{\partial A}_\rho \leq \sum_{y\in S,z\not\in S}
\abs{\Lag_{y,z}(\psi)}_\rho 
\leq 2 C_\nabla\,\abs{\partial S}_{w}.
\end{equation}
We now need to bound the denominator of the Cheeger constant from
above, which requires us to control the weighted degrees
$\d_y$. Note that
\begin{equation}\label{eq:dyy}
\d_y 
=\sum_{z\neq y} w_{yz}
\leq \frac{1}{\eps_{tw}}\sum_{z\neq y} | \Lag_{y,z}(\psi)|_{\rho} 
\leq \frac{1}{\eps_{tw}}|\partial \Lag_y(\psi)|_{\rho},
\end{equation}
where the second inequality comes from the fact that the facets
$\Lag_{y,z}(\psi)$ form a partition of the boundary $\partial
\Lag_y(\psi) \cap \interior{X}$ up to a $\Haus^{d-1}_g$--negligible
set. To see that this is the case, it suffices to remark that in the
exponential chart of $y$, the intersection of two distinct facets
adjacent to $y$ has a finite $\Haus^{d-2}_g$--measure, as implied by
Proposition~\ref{prop:HausSigma}.

In order to apply the (continuous) Cheeger inequality, we need to
replace the weighted area of the boundaries of Laguerre cells in \eqref{eq:dyy} by the
weighted volume of the cells. We have
\begin{equation*}
\begin{aligned}
  \Haus^{d-1}_g(\partial \Lag_y(\psi)) &\leq \biLip^{d-1} \Haus^{d-1}((\exp_y^c)^{-1} \partial \Lag_y(\psi)) \\
  &\leq \biLip^{d-1} \Haus^{d-1}(\partial X_y) \\
  &\leq \biLip^{2(d-1)} \Haus^{d-1}_g(\partial X).
\end{aligned}
\end{equation*}
The first and third inequalities use the definition of the
bi-Lipschitz constant $\biLip$ of the exponential map, while the second
inequality uses the monotonicity of the $\Haus^{d-1}$--measure of the
boundary of a convex set with respect to inclusion (see \cite[p.211]{schneider}). Using the
assumption $\abs{\Lag_y(\psi)}_\rho \geq \eps$, this gives us a (rather
crude) reverse isoperimetric inequality
\begin{equation}
  \begin{aligned}
  \abs{\partial \Lag_y(\psi)}_\rho &\leq \nr{\rho}_{\infty} \Haus^{d-1}_g(\partial \Lag_y(\psi))\\
  &\leq   \frac{\nr{\rho}_{\infty}}{\eps}\biLip^{2(d-1)} \Haus^{d-1}_g(\partial X) \abs{\Lag_y(\psi)}_\rho
  \end{aligned}
  \label{eq:Hausbd}
\end{equation}
We remark here that the above inequality is never sharp, see also Remark \ref{rem:nonsharpnessofisoperimetric}.
Combining
\eqref{eq:dyy}, \eqref{eq:Hausbd} and $\abs{A}_\rho = \sum_{y\in S}
\abs{\Lag_y(\psi)}_\rho$ we obtain
$$ \abs{S}_{w} = \sum_{y\in S} \d_y \leq \frac{1}{\eps}
\frac{\nr{\rho}_{\infty} \biLip^{2(d-1)} \Haus^{d-1}_g(\partial
  X)}{\eps_{tw}} \abs{A}_\rho.
$$ The same inequality holds for the complement $\abs{X\setminus
  S}$. We combine the previous inequality with Equation
\eqref{eq:cheeger-numerator} and with Lemma~\ref{lem:PWtoCheeger} to
get a lower bound on the Cheeger constant
\begin{equation}\label{eq:minoration_cheeger_constant}
\h(w) \geq 
\frac{\eps_{tw}\eps}{\biLip^{2(d-1)} C_\nabla \Haus^{d-1}_g(\partial X) \nr{\rho}_{\infty} \Cpw}.
\end{equation}
Note that, in order to apply Lemma~\ref{lem:PWtoCheeger} we
implicitly used the fact that $A$ is a Lipschitz domain (as a finite
union of Lipschitz domains, see Lemma~\ref{lem:LipDomain}) whose
boundary has finite $\Haus^{d-1}_g$--measure (by
Equation~\eqref{eq:Hausbd}).

\subsubsection{Step 2} In order to apply the Cheeger inequality, we still need to  bound from below the weighted degree  $d_y$. By \eqref{eq:cmpWrho} one has, using the crucial fact that $\abs{\partial \Lag_{y}(\psi)}_\rho$ is the measure of $\partial \Lag_y(\psi)\cap \interior{X}$,
$$ d_y = \sum_{z\neq y} w_{yz} \geq
\frac{1}{2C_\nabla} \sum_{z\neq y} \abs{\Lag_{y,z}(\psi)}_\rho \geq
\frac{1}{2C_\nabla} \abs{\partial \Lag_{y}(\psi)}_\rho.
$$
Taking $A=\Lag_{y}(\psi)$ in the definition of the weighted Cheeger constant $\h(\rho)$ defined in Lemma \ref{lem:PWtoCheeger}, one gets 
$$
\abs{\partial \Lag_{y}(\psi)}_\rho \geq \h(\rho)\ \min( \abs{\Lag_{y}(\psi)}_\rho , \abs{ X \setminus \Lag_{y}(\psi)}_\rho )
\geq \h(\rho) \varepsilon.
$$ The last inequality comes from the assumption that each Laguerre
cell has a mass greater than $\varepsilon$ and that $X \setminus
\Lag_{y}(\psi)$ also contains a Laguerre cell (except for the trivial
case where $Y$ is a singleton). We deduce
\begin{equation}
d_y \geq \frac{\eps}{C_\nabla\Cpw}. \label{eq:dyylb}
\end{equation}

\subsubsection{Step 3}  Combining the Cheeger inequality
 with Equation \eqref{eq:minoration_cheeger_constant} and
 \eqref{eq:dyylb} we have $\lambda(w)\geq C\eps^3$ where
\begin{equation}
C := 
 \frac{\eps_{tw}^2}{2\biLip^{4(d-1)} C_\nabla^3 \left(\Haus^{d-1}_g(\partial X)\right)^2
   \nr{\rho}^2_{\infty} \Cpw^3}.\label{eq:StrongConcavityConstant}
\end{equation} Since the graph induced by the Hessian is connected, the kernel of 
$- \D^2\Phi(\psi)$ is equal to the space of constant functions over
$Y$, implying that $\Ker(- \D^2\Phi(\psi))^\perp = E_Y$.  Then, using
the variational characterization of the first nonzero eigenvalue of
the Laplacian matrix we get:
 \begin{equation}
   C\eps^3 \leq \lambda(w) = \min_{v \in E_Y} \frac{\sca{-D^2\Phi(\psi)}{v}}{\nr{v}^2}. \tag*{\qed}
 \end{equation}

\section{Convergence of the damped Newton algorithm}
\label{sec:Newton}
The goal of this section is to show the convergence of the Damped
Newton algorithm for semi-discrete optimal transport. This follows in
fact from a more general result. We establish in Section
\ref{section:damped-newton-general} the convergence of the damped
Newton algorithm (Algorithm~\ref{algo:newton}) under general
assumptions on the functional. We finally apply this algorithm to the
semi-discrete optimal transport problem, using the intermediate
results (regularity and strict concavity of the Kantorovich
functional) proven in
Section~\ref{sec:Regularity} and \ref{sec:Concavity}.

\subsection{General damped Newton algorithm}\label{section:damped-newton-general}
Let $Y$ be a finite set and denote $\Rsp^Y$ the space of functions on
$Y$. We consider $\Prob(Y)$, the space of probability measures on $Y$, 
as a subset of $\Rsp^Y$. Finally, we denote by $E_Y$ the space of
functions on $Y$ who sum to zero. In this section, we show that
Algorithm~\ref{algo:newton} can be used to solve non-linear equations
$G(\psi) = \mu$ where $\mu \in \Prob(Y)$ and the map $G: \Rsp^Y\to
\Prob(Y)$ satisfies some regularity and monotonicity assumptions.



\begin{proposition}\label{prop:newton}
  Let $G$ be a functional from $\Rsp^Y$ to $\Prob(Y)$ which is
  invariant under addition of a constant. Let $G(\psi) = \sum_{y\in
    Y} G_y(\psi) \one_y$ and
  $$\K^\eps = \{ \psi \in \Rsp^Y \mid \forall y\in Y,~ G_y(\psi) \geq
  \eps \},$$ and assume
  that $G$ has the following properties:
  \begin{itemize}
  \item[(i)] (Regularity) For every positive $\eps$, $G$ is
    $\Class^{1,\alpha}$ on $\K^\eps$. We let $L_\eps$ be the smallest
    constant such that
$$
    \forall \phi\neq\psi\in \K^\eps,~~ \frac{\nr{G(\phi)-G(\psi)}}{\nr{\phi - \psi}} +
    \frac{\nr{\D G(\phi)-\D G(\psi)}}{\nr{\phi - \psi}^\alpha}\leq L_\eps.
    $$
    \item[(ii)] (Uniform monotonicity) For every $\eps > 0$, there exists a positive constant $\kappa_\eps$ such that $G$ is $\kappa_\eps$--uniformly monotone on
      $\K^\eps \cap E_Y$:
      $$\forall \psi \in \K^\eps, \forall v\in E_Y,~~  \sca{v}{\D G(\psi)v}\geq \kappa_\eps \nr{v}^2.$$
  \end{itemize}
  Now, let $\mu \in \Prob(Y)$ and let $\psi_0$ be a function on $Y$
  such that the constant $\eps_0$ defined in \eqref{eq:nonzerocells}
  is positive. Set $\kappa := \min(\kappa_{\frac{1}{2}\eps_0},1)$ and $L :=
  \max(L_{\frac{1}{2}\eps_0},1)$. Then, the iterates $(\psi_k)$ of
  Algorithm~\ref{algo:newton} satisfy
  \begin{align}
    &\nr{G(\psi_{k+1}) - \mu} \leq (1 - \bar{\tau}_k/2)
    \nr{G(\psi_{k}) - \mu} \notag \\
    & \hbox{ where } \bar{\tau}_k := \min\left(
\frac{\kappa^{1+\frac{1}{\alpha}}\eps}{d^{\frac{1}{\alpha}}L^{\frac{1}{\alpha}}
  \nr{G(\psi_k) - \mu}}, 1\right). \label{eq:taubark}
  \end{align}
  In addition, as soon as $\bar{\tau}_k = 1$ one has
  $$ \nr{G(\psi_{k+1}) - \mu} \leq \frac{L\nr{G(\psi_{k}) -
      \mu}^{1+\alpha}}{\kappa^{1+\alpha}}.  $$ In particular, the
  damped Newton's algorithm converges globally with linear speed and
  locally with superlinear speed (quadratic speed if $\alpha=1$).
  \end{proposition}

\begin{proof}
  We set $\eps := \eps_0$, $L := \max(L_{\eps/2},1)$ and $\kappa :=
  \min(\kappa_{\eps/2},1)$.  First, we remark that for every $\psi \in
  \K^{\frac{\eps}{2}}$, the pseudo-inverse $\D G^+(\psi)$ maps the
  subspace $E_Y$ to itself. The uniform monotonicity of $G$ therefore
  implies that $\nr{\D G^+(\psi)} \leq 1/\kappa$, where $\nr{\cdot}$
  is the operator norm on $\Rsp^Y$.

We start by the analysis of a single iteration of the algorithm. We
let $\psi := \psi_k \in \K^{\eps}$, define $v := \D G(\psi)^+ (G(\psi)
- \mu)$ and $\psi_\tau:=\psi - \tau v$. Since the pseudo-inverse $\D
G^+(\psi)$ is $1/\kappa$-Lipschitz, one has $\nr{v} \leq \nr{G(\psi) -
  \mu}/\kappa$. Now let $\tau_1$ be the largest time before the curve
$\psi_\sigma$ leaves $\K^{\eps/2}$.  In particular, $\psi_{\tau_1}$
lies at the boundary of $\K^{\eps/2}$, meaning that there must exist a
point $y$ in $Y$ such that $G_y(\psi_{\tau_1}) = \eps/2$. This implies
that $\nr{G(\psi_{\tau_1}) - G(\psi)} \geq \eps/2$, and using the
Lipschitz bound on $G$ we obtain a lower bound on $\tau_1$
\begin{align*}
\frac{\eps}{2} \leq \nr{G(\psi_{\tau_1}) -  G(\psi)} \leq L \tau_1  \nr{v} \leq 
 \frac{L\tau_1}{\kappa} \nr{G(\psi) - \mu}.
\end{align*}
This implies that $\tau_1$ is necessarily larger than $\kappa
\eps/(2 L\nr{G(\psi) - \mu})$.  We now established that the curve $\tau
\mapsto \psi_\tau$ remains in $\K^{\eps/2}$ before time $\tau_1$,
implying that the function $\tau \in [0,\tau_1] \mapsto G(\psi_\tau)$
is uniformly $\Class^{1,\alpha}$. Applying Taylor's formula we get
\begin{equation}
  \label{eq:DLreste}
  G(\psi_\tau) = G(\psi - \tau \D G(\psi)^+ (G(\psi)
  - \mu)) = (1-\tau) G(\psi) + \tau \mu + R(\tau)
\end{equation}
where, using $v = \D G(\psi)^+ (G(\psi) - \mu)$, and the $\alpha$-H\"older property for $\D G$
\begin{align}
  \nr{R(\tau)} &= \nr{\int_{0}^\tau (\D G(\psi_\sigma) - \D G(\psi)) v   \dd \sigma} \notag \\
  &\leq \frac{L}{\alpha+1} \tau^{\alpha+1} \nr{v}^{\alpha+1} \leq
  \frac{L \nr{G(\psi) - \mu}^{(1+\alpha)}}{\kappa^{(1+\alpha)}} \tau^{(1+\alpha)} \label{eq:Rtau}
\end{align}
For every $y\in Y$, using that $\mu_y \geq
2\eps$ (by  \eqref{eq:nonzerocells}) and $G_y(\psi) \geq \eps$, one gets
$$G_y(\psi_\tau) \geq (1-\tau) G_y(\psi) + \tau \mu_y + R_y(\tau) \geq
(1+\tau) \eps - \nr{R(\tau)}.$$ If $\tau$ is chosen such that such
that $\nr{R(\tau)} \leq \tau\eps$ we will have $G_y(\psi_\tau) \geq
\eps$ for all points $y$ in $Y$ and therefore $\psi_\tau$ will belong
to $\K^\eps$. Thanks to our estimate on $R(\tau)$ this will be true
provided that
\begin{equation*}
  \tau \leq \tau_2:=\min\left(\tau_1, \frac{\kappa^{1+\frac{1}{\alpha}}\varepsilon^{\frac{1}{\alpha}}}
       {L^{\frac{1}{\alpha}}\|G(\psi)-\mu\|^{1+\frac{1}{\alpha}}}\right).
\end{equation*}
Finally we establish the second inequality required by Step 2 of the
Algorithm. To do that, we subtract $\mu$ from both sides in 
\eqref{eq:DLreste} to obtain
\begin{equation}
  G(\psi_\tau) - \mu = (1-\tau) (G(\psi) - \mu) + R(\tau). \label{eq:Gmureste}
\end{equation}
In order to get $\nr{G(\psi_\tau) - \mu} \leq
(1-\frac{\tau}{2})\nr{G(\psi) - \mu}$, it is sufficient to establish
that $\nr{R(\tau)} \leq \frac{\tau}{2}\nr{G(\psi) - \mu}$. Using the
estimation on $\nr{R(\tau)}$ again, we see that it suffices to take
$$ \tau \leq
\tau_3:=\min\left(\tau_2,\frac{\kappa^{1+\frac{1}{\alpha}}}{L^{\frac{1}{\alpha}}\|G(\psi)-\mu\|2^{\frac{1}{\alpha}}},1\right).$$
Finally, using $L\geq 1$, $\kappa \leq 1$ and $\nr{G(\psi) - \mu} \leq
d$ (since $G(\psi)$ and $\mu$ are probability measures), we can
establish that $\tau_3 \geq \bar{\tau}_k$ where the value of
$\bar{\tau}_k$ is defined in \eqref{eq:taubark}. This ensures the
first estimate on the improvement of the error between two successive
steps.

By this first estimate, we see that there exists $k_0$ such that
$\bar{\tau}_k = 1$ for $k\geq k_0$. When this happens, one can use
\eqref{eq:Gmureste} to get $\nr{G(\psi_{k+1}) - \mu} \leq
\nr{R(\tau)}$. We obtain the second estimate of the theorem by
plugging in \eqref{eq:Rtau}.
\end{proof}

\subsection{Proof of Theorem \ref{th:main}}
Proposition \ref{prop:newton} can be directly applied to the gradient
of the Kantorovich functional, or more precisely to
$$G(\psi) := \sum_{y\in Y} \rho(\Lag_y(\psi)) \one_y = \nabla
\Phi(\psi) + \mu$$ In that case, the set $K^\eps$ is given by
$$ K^\eps = \{ \psi \in \Rsp^Y \mid \forall y\in Y,~
\rho(\Lag_y(\psi)) \geq \eps \}. $$ We have assumed that the
probability density $\rho$ is $\Class^{0,\alpha}(X)$ where $X$ is a
$c$-convex, compact subset of $\Omega$. Then, by
Theorem~\ref{th:StrongRegularity}, for any $\eps > 0$, the map $G$ is
uniformly $\Class^{1,\alpha}$ over $\K^\eps$. This ensures that the
(Regularity) condition of Proposition \ref{prop:newton} is
satisfied. Furthermore, since we also assumed that $\rho$ satisfies a
weighted Poincaré-Wirtinger inequality, we can apply Theorem
\ref{th:UniformConcavity} to see that the (Uniform monotonicity)
hypothesis of Proposition~\ref{prop:newton} is also
satisfied. Applying Proposition~\ref{prop:newton}, we deduce the
desired convergence rates for Algorithm~\ref{algo:newton}.

\subsection{Numerical results}
We conclude the article with a numerical illustration of this
algorithm, for the cost $c(x,y) = \nr{x - y}^2$ and for a
piecewise-linear density. The source density is piecewise-linear over
a triangulation of $[0,3]$ with $18$ triangles (displayed in
Figure~\ref{fig:numerical}). It takes value $1$ on the boundary
$\partial [0,3]^2$ and vanishes on the square $[1,2]^2$. In
particular, the support of this density is not simply connected and
not convex. The target measure is uniform over a uniform grid
$\frac{1}{n-1} \{0,\hdots,n-1\}^2$. Figure~\ref{fig:numerical}
displays the iterates of the Newton algorithm, which in this case
takes 25 iterations to solve the optimal transport problem with an
error equal to the numerical precision of the machine. The source code
of this algorithm is publicly
available\footnote{\url{https://github.com/mrgt/PyMongeAmpere}}.

We finally note that recent progress in computational geometry would
allow one to implement Algorithm~\ref{algo:newton} for the quadratic
cost on $\Rsp^3$, refining \cite{levy2014numerical} or
\cite{de2015power}. It should also be possible to deal with optimal
transport problems arising from geometric optics, such as the
far-field reflector problem \cite{de2014intersection}, whose
associated cost satisfies the Ma-Trudinger-Wang condition
\cite{loeper2011}.

\begin{figure}
  \begin{center} \label{fig:numerical}
    \includegraphics[width=.4\textwidth]{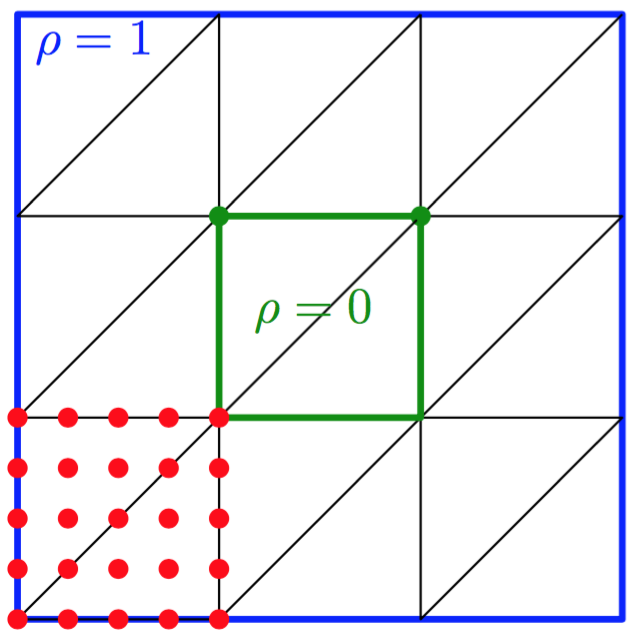}
    \\
  \begin{tabular}{ccc}
     \includegraphics[width=.3\textwidth]{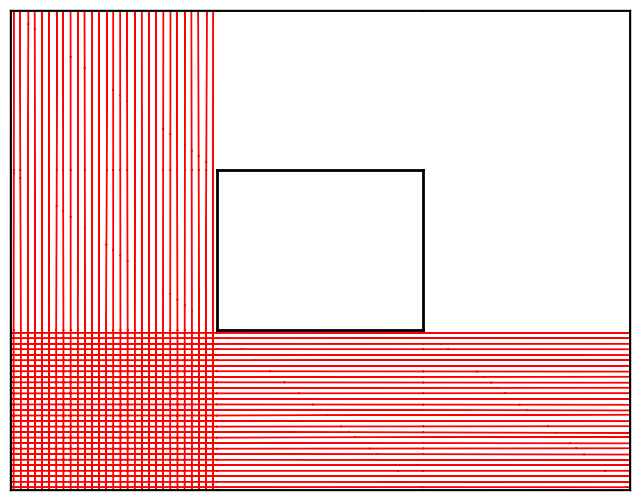} &
     \includegraphics[width=.3\textwidth]{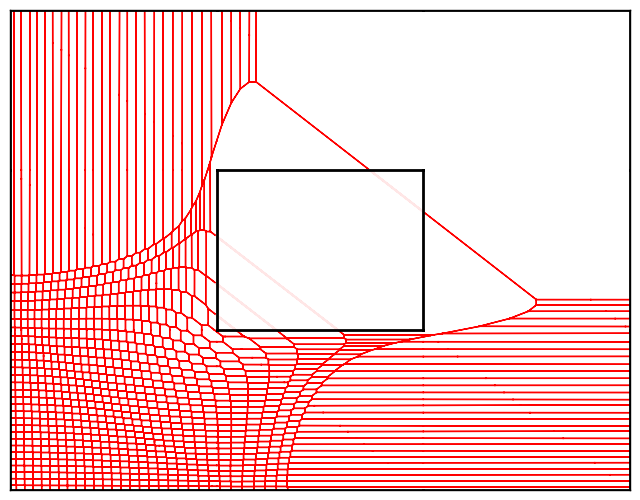} &
     \includegraphics[width=.3\textwidth]{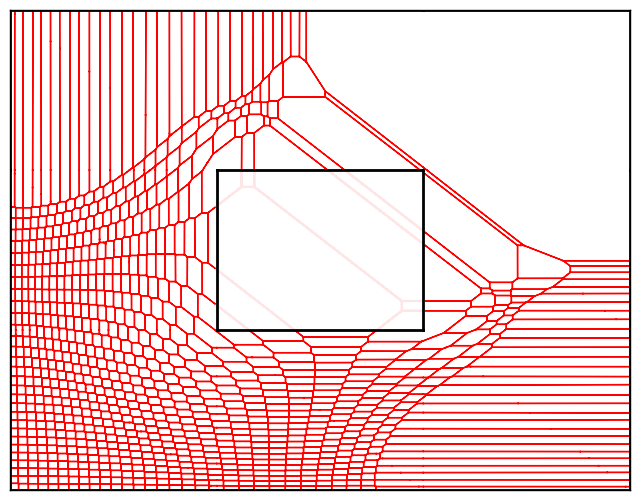} \\
     \includegraphics[width=.3\textwidth]{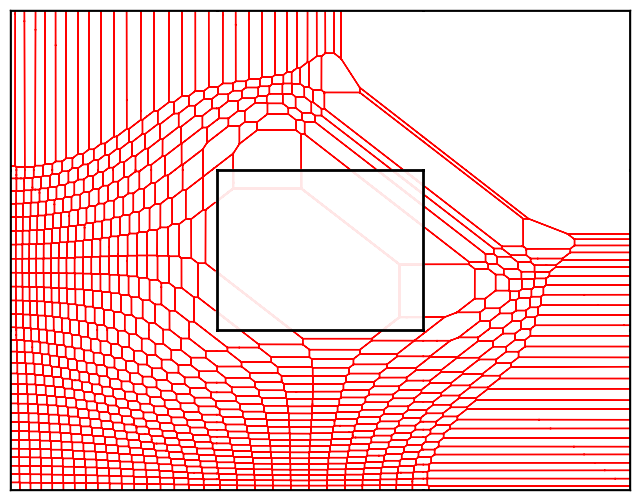} &
     \includegraphics[width=.3\textwidth]{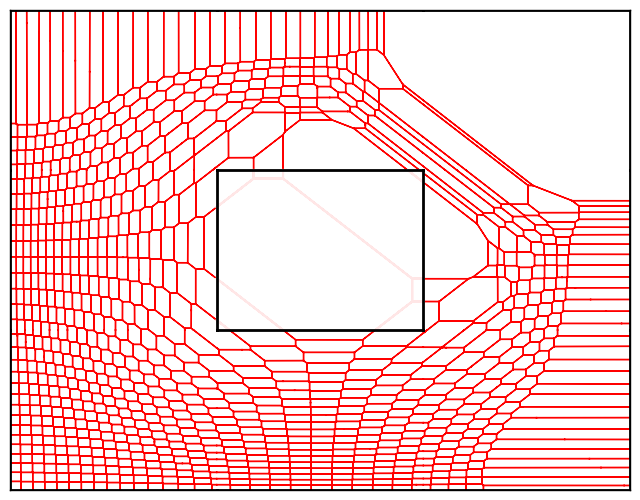} &
     \includegraphics[width=.3\textwidth]{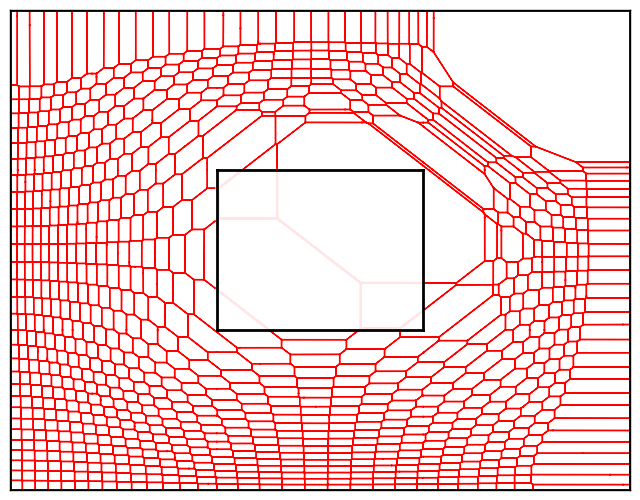} \\
     \includegraphics[width=.3\textwidth]{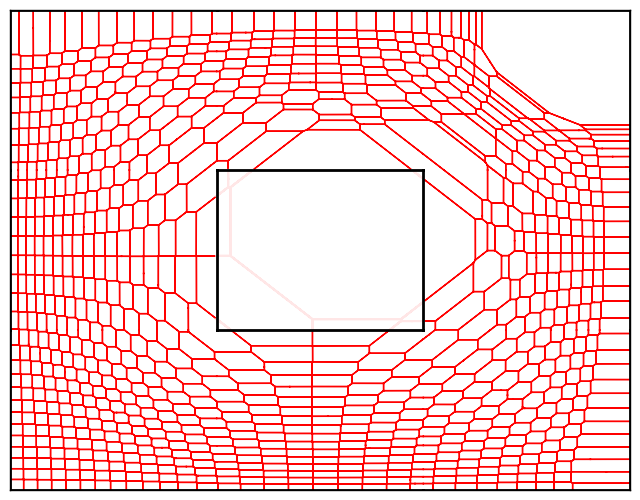} &
     \includegraphics[width=.3\textwidth]{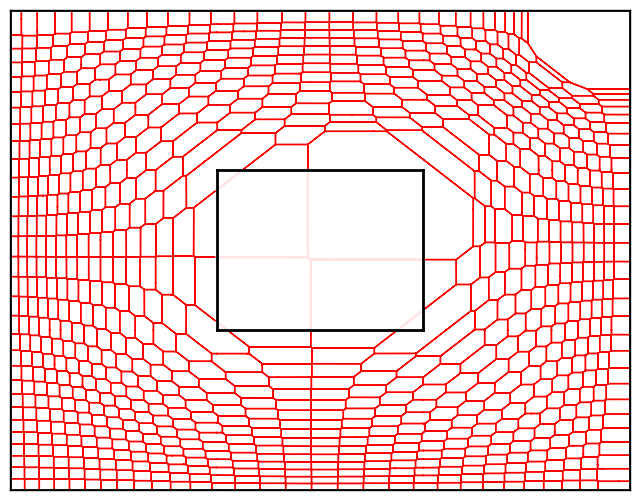} &
     \includegraphics[width=.3\textwidth]{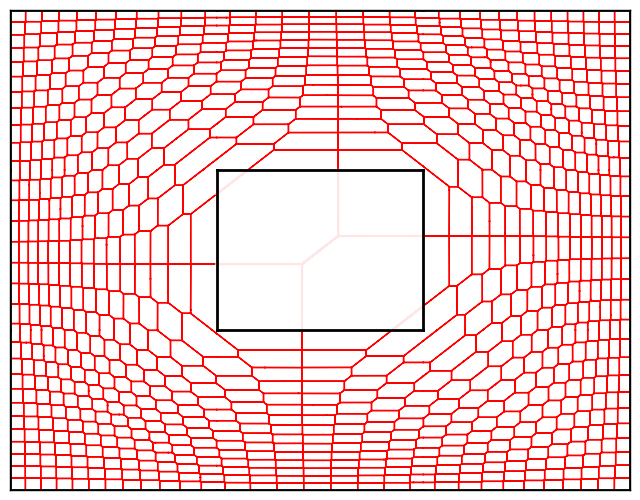} \\
  \end{tabular}
  \caption{Evolution of the Laguerre cells during the execution of the
    damped Newton algorithm for semi-discrete optimal transport. (Top)
    The source density $\rho$ is piecewise linear over the domain
    $X=[0,3]^3$ over the displayed triangulation: it takes value $1$
    on the boundary of the square $[0,3]^2$ and $0$ on the boundary of
    $[1,2]^2$. The target measure is uniform over a $30^2$ uniform
    grid in $[0,1]^2$. (Bottom) Laguerre cells at steps
    $k=0,2,6,9,12,15,18,21$ and $25$.}
  \end{center}
 \end{figure}

\appendix
\section{A weighted Poincar\'e-Wirtinger inequality}
\label{app:poincare}
In this section, we provide an (almost) explicit example of a
probability density on $\Rsp^d$ whose support is an annulus, therefore
not simply connected, but which still satisfies a weighted
Poincaré-Wirtinger inequality.
\begin{proposition}
Let $0 < r < R$ and assume that $\bar{\rho} \in \Class^0([0,R])$ is a
probability density with $\bar{\rho} = 0$ on $[0,r]$ and $\bar{\rho}$
concave on $[r,R]$. Consider
$$ \rho(x) = \frac{1}{\nr{x}^{d-1} \omega_{d-1}} \bar{\rho}(\nr{x})
\qquad \hbox{over} \qquad X := \B(0,R)\subseteq \Rsp^d,$$ where
$\omega_{d-1}$ is the volume of the unit sphere $\Sph^{d-1}$. Then,
$\rho$ satisfies the weighted Poincaré-Wirtinger inequality
\eqref{eq:L1-poincare} for some positive constant.
\end{proposition}

The proof relies on two $L^1$-Poincar\'e-Wirtinger inequalities. The
first inequality is the usual Poincar\'e-Wirtinger inequality on the
sphere: given a $\Class^1$ function $f$ on $\Sph^{d-1}$, and $F_{d-1}
:= (1/\omega_{d-1}) \int_{\Sph^{d-1}} f(z) \dd z$,
  \begin{equation}
    \int_{\Sph^{d-1}} \abs{f(z) - F_{d-1}} \dd \Haus^{d-1}(z)
    \leq c_d\int_{\Sph^{d-1}} \nr{\nabla f(z)} \dd\Haus^{d-1}(z)
    \label{eq:poincare-sphere}
  \end{equation}
for some positive constant $c_d$. The second inequality is a
Poincaré-Wirtinger inequality on the segment $[0,R]$ weighted by
$\bar{\rho}$. Given a function in $\Class^1([0,R])$, and letting $F_1
:= \int_{0}^{R} f(r) \bar{\rho}(r)\dd r / \int_{0}^{R}
\bar{\rho}(r)\dd r$,
  \begin{equation}
    \int_0^R \abs{\bar{f}(r) - F_1} \bar{\rho}(r) \dd r \leq
    c_{\bar{\rho}} \int_{0}^{R} \abs{\bar{f}'(r)} \bar{\rho}(r) \dd r
        \label{eq:poincare-segment}
  \end{equation}
  for some positive constant $c_{\bar{\rho}}$ depending only on
  $\bar{\rho}$. The inequality \eqref{eq:poincare-segment} can be
  deduced from Theorem~2.1 in \cite{acosta2004optimal} and from the
  concavity of $\bar{\rho}$ on $[r,R]$.
    
\begin{proof}  
We now proceed to the proof of the Poincaré-Wirtinger inequality for
$(X,\rho)$.  Let $f:\B(0,R)\to\Rsp$ be a function of class
$\Class^1$. By polar coordinates and the definition of $\rho$, one
has
\begin{align*}
F&:=\int_{\B(0,R)} f(x) \rho(x) \dd \Haus^d(x) \\
&= \int_0^R 
\frac{1}{\omega_{d-1} r^{d-1}}\int_{\Sph^{d-1}(r)} f(z) \bar{\rho}(r)
\dd \Haus^{d-1}(z) \dd r 
= \int_{0}^R \bar{f}(r) \bar{\rho}(r) \dd r ,
\end{align*}
where the function $\bar{f}(r)$ is the mean value of $f$ over the
sphere $\Sph^{d-1}(r)$,
$$ 
\bar{f}(r) = \frac{1}{\omega_{d-1}r^{d-1}} \int_{\Sph^{d-1}(r)}
f(z)  \dd\Haus^{d-1}(z) = \frac{1}{\omega_{d-1}} \int_{\Sph^{d-1}} f(rz) \dd\Haus^{d-1}(z).
$$
Using the triangle inequality and the relation between $\bar{\rho}$ and $\rho$ we get
\begin{align}
 & \int_{\B(0,R)} \abs{f(x) - F} \rho(x) \dd \Haus^d(x)  
 = \int_0^R \int_{\Sph^{d-1}(r)} \abs{f(z) - F} \rho(z)
  \dd\Haus^{d-1}(z) \dd r \notag \\
 & \quad \leq \int_0^R \bar{\rho}(r)\abs{\bar{f}(r) - F} \dd r  + \int_0^R \frac{ \bar{\rho}(r)}{r^{d-1} \omega_{d-1}}
  \int_{\Sph^{d-1}(r)}\abs{f(z) - \bar{f}(r)}  \dd \Haus^{d-1}(z) \dd r \label{eq:poincare:split}
\end{align}
We first deal with the second term of the sum. Using the
Poincar\'e-Wirtinger inequality on the sphere
\eqref{eq:poincare-sphere}, we have
$$
\int_{\Sph^{d-1}(r)}\abs{f(z) -
    \bar{f}(r)} \dd \Haus^{d-1}(z)
  \leq c_d \int_{\Sph^{d-1}(r)} \nr{\restr{\nabla f(z)}_{z^\perp}} \dd\Haus^{d-1}(z),
  $$ where $\restr{\nabla f(z)}_{z^\perp}$ is the orthogonal
  projection of the gradient on the tangent plane $\{z\}^\perp$, so
  that
\begin{equation}
 \int_0^R \frac{\bar{\rho}(r)}{r^{d-1} \omega_{d-1}}
\int_{\Sph^{d-1}(r)}\abs{f(z) - \bar{f}(r)} \dd \Haus^{d-1}(z) \leq c_d
\int_{B(0,r)} \nr{\restr{\nabla f(x)}_{x^\perp}} \rho(x)\dd \Haus^d(x).
       \label{eq:partial1}
\end{equation}
By the calculation of $F$ above, we see $F$ is also the mean value of
$\bar{f}$ weighted by $\bar{\rho}$. We can therefore control the first
term of the upper bound of \eqref{eq:poincare:split} using the
Poincar\'e-Wirtinger inequality on the segment
\eqref{eq:poincare-segment}:
$$\int_0^R \bar{\rho}(r) \abs{\bar{f}(r) - F} \dd r \leq
c_{\bar{\rho}} \int_{0}^{R} \abs{\bar{f}'(r)}   \bar{\rho}(r) \dd r.
$$
Now, notice that:
$$ \bar{f}'(r) = \lim_{h\to 0} \frac{\bar{f}(r+h) - \bar{f}(r)}{h} = \lim_{h\to 0}
\frac{1}{\omega_{d-1}} \int_{\Sph^{d-1}} \frac{f((r+h) z) - f(rz)}{h} \dd \Haus^{d-1}(z), $$
from which we deduce
$$ \abs{\bar{f}'(r)} \leq \frac{1}{\omega_{d-1}} \int_{\Sph^{d-1}} \abs{\frac{\partial f}{\partial r}(rz)} \dd \Haus^{d-1}(z) = \frac{1}{\omega_{d-1} r^{d-1}} \int_{\Sph^{d-1}(r)} \abs{\sca{\nabla f(z)}{\frac{z}{r}}} \dd \Haus^{d-1}(z)$$
Integrating this inequality shows that
\begin{align}
 \int_0^R \bar{\rho}(r) \abs{\bar{f}(r) - F} \dd r &\leq c_{\bar{\rho}} \int_0^R \frac{\bar{\rho}(r)}
     {\omega_{d-1} r^{d-1}} \int_{\Sph^{d-1}(r)} \abs{\sca{\nabla f(z)}{\frac{z}{r}}} \dd \Haus^{d-1}(z) \notag \\
     &= c_{\bar{\rho}}\int_{\B(0,R)} \abs{\sca{\nabla f(x)}{\frac{x}{\nr{x}}}} \rho(x) \dd \Haus^d(x).
     \label{eq:partial2}
\end{align}
From the simple inequality $(a+b)^2  \leq 2(a^2 + b^2)$, we get
$$\abs{\sca{\nabla f(x)}{\frac{x}{\nr{x}}}} + \nr{\restr{\nabla f(x)}_{x^\perp}} \leq \sqrt{2}
\nr{\nabla f(x)}. $$
Using the bounds \eqref{eq:partial1} and \eqref{eq:partial2} in Equation \eqref{eq:poincare:split}, we get the desired inequality:
\begin{equation*}
  \int_{\B(0,R)} \abs{f(x) - F} \rho(x) \dd \Haus^d(x) \leq \sqrt{2} (c_d + c_{\bar{\rho}}) \int_{\B(0,R)} \nr{\nabla f(x)} \rho(x) \dd \Haus^d(x).\qedhere
  \end{equation*}

\end{proof}

\section{Proof of Theorem~\ref{th:LocReg}}
\label{app:proof-locreg}

\subsection{Existence of partial derivatives}
\label{sec:existence-PD}
Without loss of generality, we assume that $\bm{\lambda}_0 = \bm{0}$.
We start the proof of Theorem~\ref{th:LocReg} by showing the existence
of partial derivatives for the map $\Gloc$. In this section, we denote
$e_1,\hdots,e_N$ the canonical basis of $\Rsp^N$. We start by
rewriting the finite difference defining the partial derivative of
$\Gloc$ in direction $e_i$ using the coarea formula. Fix $\nr{\bm{\lambda}}<T_{tr}$. For $t>0$, one
has:
\begin{align}
  \frac{1}{t} (\Gloc(\bm{\lambda} + t\bm{e}_i) - \Gloc(\bm{\lambda})) &=
  \frac{1}{t}  \int_{K(\bm{\lambda} + t \bm{e}_i) \setminus K(\bm{\lambda})} \rholoc(x)
  \dd \Haus^d(x) \notag \\
  &=  \frac{1}{t}  \int_{\lambda_i}^{\lambda_{i}+t}
  \gloc(s)
  \d s,
\label{eq:G:fd}
\end{align}
where the function $\gloc$ is defined by 
\begin{equation}
\gloc(s) := \int_{\cap_{j\neq i}K_j(\lambda_j) \cap f_i^{-1}(s)}
\frac{\rholoc(x)}{\nr{\nabla f_i(x)}} \dd \Haus^{d-1}(x).
\label{eq:G:pd}
\end{equation}
The same reasoning also holds for $t<0$. We now claim that $\gloc$ is
continuous on some interval
around $\lambda_i$, which by \eqref{eq:G:fd} and the Fundamental
Theorem of Calculus will imply that the limit as $t\to 0$ in
\eqref{eq:G:fd} exists and is equal to $\gloc(\bm{\lambda})$, thus
establishing the formula \eqref{eq:PD}.  The continuity of $\gloc$ follows from the next proposition, which
is formulated in a slightly more general way.

\begin{proposition}\label{prop:PDUC}
  Let $\sigma$ be a continuous non-negative function on $\Xloc$ and let
  $\omega$ be the modulus of continuity of $\sigma$. Given any vector
  $\bm{\lambda}$ in $\Rsp^N$ with $\nr{\bm{\lambda}}_{\infty} \leq T_{tr}$,
  consider the function
  \begin{align*} h: s\in \Rsp\mapsto \int_{L\cap S_s} \sigma(x) \dd \Haus^{d-1}(x), 
\end{align*} 
where $L := \bigcap_{j\neq i} K_j(\lambda_j)$ and $S_s :=
f_{i}^{-1}(s)$. Then $h$ is uniformly continuous on
$[-T_{tr},T_{tr}]$ and has modulus of continuity
\begin{equation}
  \omega_h(\delta) = 
  C_1\cdot(\omega(C_2\delta) + \abs{\delta}),
  \label{eq:PDUC}
  \end{equation} where the constants
only depend on $\nr{f_i}_{\Class^{1,1}}$, $\diam(\Xloc)$, $\eps_{nd}$,
$\eps_{tr}$ and $\nr{\sigma}_\infty$.
\end{proposition}

Taking $\sigma = \rholoc/\nr{\nabla f_i}$ in the previous proposition,
which is continuous using the non-degeneracy condition \eqref{eq:ND}
and the assumption $f_i\in \Class^{1,1}(\Xloc)$, we see that the function
$\gloc$ defined by \eqref{eq:G:pd} is continuous. This implies the
existence of partial derivatives and establishes formula
\eqref{eq:PD}.  The proof of Proposition~\ref{prop:PDUC} requires the following lemma.

\begin{lemma} \label{lemma:Phi} Assume that the functions $f_i: \Xloc \to \Rsp$
  satisfy \eqref{eq:ND}. Then, for every $i \in \{1,\hdots,N\}, $
  there exists a map $\Phi_i: \Xloc \times \Rsp \to \Rsp^d$ such that:
\begin{enumerate}[(i)]
\item For any $(x,t)$ in $\Xloc \times \Rsp$ such that the curve $\Phi_i(x,
  [0,t])$ remains in $\Xloc$, one has $f_i(\Phi_i(x, t)) = f_i(x)+t$.
\item The map $\Phi_i$ satisfies the following inequalities for every $x,y\in \Xloc$, $t\in\Rsp$:
\begin{align}
  & \nr{\Phi_i(x,t) - \Phi_i(x, s)} \leq \frac{\abs{t-s}}{\eps_{nd}}  \label{ineq:LipTime}\\
  & \nr{\Phi_i(x,t) - \Phi_i(y,t)} \leq \exp(C_\Phi\abs{t}) \nr{x - y},
\label{ineq:Gronwall}
\end{align}
where $C_\Phi := 3C_L/\eps_{nd}^2$.
\end{enumerate}
\end{lemma}

\begin{proof} We consider the vector field $V_i^0(x) = \nabla
  f_i(x)/\nr{\nabla f_i(x)}^2$ on $\Xloc$, which satisfies
  $\nr{V_i^0}_{\infty} \leq 1/\eps_{nd}$ and whose Lipschitz constant
  is bounded by $C_\Phi$.
  This vector field is extended
  to $\Rsp^d$ using the orthogonal projection on $\Xloc$, denoted $p_{\Xloc}$,
  $$ \forall x\in \Rsp^d,~ V_i(x) := V^0_i(p_{\Xloc}(x)). $$ By
  convexity of $\Xloc$, the map $p_{\Xloc}$ is $1$-Lipschitz. This
  implies that the Lipschitz constant of $V_i$ is also bounded by
  $C_\Phi$.  We let $\Phi_i$ be the flow induced by this vector
  field, which exists and is for all time since $V_i$ is bounded and uniformly Lipschitz on all of $\Rsp^d$. The inequality \eqref{ineq:LipTime} follows from the
  definition of integral curves and the bound on $\nr{V_i}$. Any
  integral curve $\gamma: [0,T] \to \Rsp^d$ of $V_i$ which remains in
  $\Xloc$ satisfies
\begin{align*}
  f_i(\gamma(t)) &= f_i(\gamma(0)) + \int_{0}^t \sca{\gamma'(s)}{\nabla f_i(\gamma(s))} \dd s \\
  &= f_i(\gamma(0)) + \int_{0}^t \sca{V_i(\gamma(s))}{\nabla f_i(\gamma(s))} \dd s 
  = f_i(\gamma(0)) + t,
\end{align*}
thus establishing (i). The inequality \eqref{ineq:Gronwall} follows from the bound on
the Lipschitz constant of $V_i$ and from Gronwall's lemma.
\end{proof}

\begin{figure}
\includegraphics{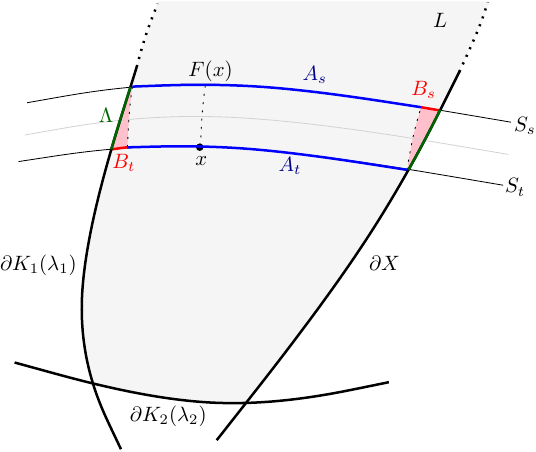}
\caption{Illustration of the proof of Proposition~\ref{prop:PDUC}.\label{fig:CoareaPartialDerivatives}}
\end{figure}

\begin{proof}[Proof of Proposition~\ref{prop:PDUC}]
  Let $t,s$ be small enough so that the transversality condition
  \eqref{eq:T} holds (that is $t,s \in [-T_{tr},T_{tr}]$). We assume
  that $t<s$ so as to fix the signs of some expressions. We consider
  the following partition of the facet $S_t \cap L$, whose geometric
  meaning is illustrated in Figure~\ref{fig:CoareaPartialDerivatives}:
  \begin{align*}
    &A_t = \{ x \in S_t \cap L \mid \Phi_i(x, [0, s-t]) \subseteq L\} \\
    &B_t = \{ x \in S_t \cap L \mid \exists u \in [0, s-t),~ \Phi_i(x, u) \in \partial L \}.
  \end{align*}
  Similarly, we define
  \begin{align*}
    &A_s = \{ x \in S_s \cap L \mid \Phi_i(x, [t-s, 0]) \subseteq L\} \\
    &B_s = \{ x \in S_s \cap L \mid \exists u \in (t-s, 0],~ \Phi_i(x, u) \in \partial L \}.
  \end{align*}
 Recall that by definition,
 \begin{equation}
   h(t) = \int_{A_t} \sigma(x)\dd\Haus^{d-1}(x) + \int_{B_t}
   \sigma(x)\dd\Haus^{d-1}(x), \label{eq:hsplit}
 \end{equation}
 where the integral is with respect to the $(d-1)$-dimensional
 Hausdorff measure. Our strategy to show the continuity of $h$ is to
 prove that the first term in the sums defining $h(t)$ and $h(s)$ in
 \eqref{eq:hsplit} are close, namely
  \begin{equation}
    \abs{\int_{A_t} \sigma(x)\dd\Haus^{d-1}(x) - \int_{A_s} \sigma(x)\dd\Haus^{d-1}(x)} \leq C_3 \cdot (\abs{s-t}+\omega(C \abs{s-t})) \label{eq:A}
  \end{equation}
  and then that the terms involving $B_t,B_s$ are small (recall
  that both sets depend on $t$ \emph{and} $s$):
  \begin{equation}
    \int_{B_t} \abs{\sigma(x)\dd\Haus^{d-1}(x)} + 
    \int_{B_s} \abs{\sigma(x)\dd\Haus^{d-1}(x)} \leq  C_4 \cdot \abs{s-t} \label{eq:B}
  \end{equation}
  The combination of the estimates \eqref{eq:A} and \eqref{eq:B}
  implies the desired inequality \eqref{eq:PDUC}. We now turn to the
  proof of these estimates, and that the constants $C_3$ and $C_4$ in
  these estimates depend on $\nr{f_i}_{\Class^{1,1}}$, $\diam(\Xloc)$,
  $\eps_{nd}$, $\eps_{tr}$ and $\nr{\sigma}_\infty$.
  \smallskip
  
  \noindent\textbf{Proof of Estimate \eqref{eq:A}.}  By
  Lemma~\ref{lemma:Phi}.(i), for any point $x$ in $A_t$ one has
  $f_i(\Phi_i(x, s-t)) = s$, so that the map $F(x) := \Phi_i(x, s-t)$
  induces a bijection between the sets $A_t$ and $A_s$.  As a
  consequence of \eqref{ineq:Gronwall}, the restriction of $F$ to
  $A_t$ is a bi-Lipschitz bijection between the sets $A_t$ and $A_s$,
  with Lipschitz constant
  $$ \max\{\nr{F^{-1}}_{\Lip(A_s)},\ \nr{F}_{\Lip(A_t)}\} \leq \exp(C_\Phi \abs{s-t}).$$
  Using a Lipschitz change of variable formula, we get
\begin{align}
  \int_{A_t} \sigma(x) \dd \Haus^{d-1}(x)
  &= \int_{F^{-1}(A_s)} \sigma(x)  \dd \Haus^{d-1}(x)  \notag \\
  &\leq \nr{F^{-1}}_{\Lip(A_s)}^{d-1} \int_{A_s} \sigma(F^{-1}(x)) \dd \Haus^{d-1}(x)  \notag \\
  &\leq \exp(C_\Phi(d-1) \abs{s-t}) \int_{A_s} \sigma(F^{-1}(x)) \dd \Haus^{d-1}(x). \label{eq:At1}
\end{align}
By definition of the modulus of continuity and thanks to \eqref{ineq:LipTime},
\begin{align*}
 \abs{\sigma(F^{-1}(x)) -\sigma(x)} &\leq \omega(\nr{F^{-1}(x) - x}) \\
&= \omega(\nr{\Phi(x,s-t) - x}) \leq \omega(\abs{s-t}/\eps_{nd}) 
\end{align*}
Integrating this inequality, we get
\begin{align}
  \int_{A_s} \sigma(F^{-1}(x)) \dd \Haus^{d-1}(x) &\leq \int_{A_s} \sigma(x) \dd \Haus^{d-1}(x) + \Haus^{d-1}(A_s) \omega(\abs{s-t}/\eps_{nd}) \notag \\
  &\leq  \int_{A_s} \sigma(x) \dd \Haus^{d-1}(x) + \Haus^{d-1}(\Xloc) \omega(\abs{s-t}/\eps_{nd}) \label{eq:At2}
  \end{align}
where the second inequality uses the monotonicity of the
$(d-1)$-dimensional Hausdorff measure of the boundary of a convex set
with respect to inclusion, see \cite[p.211]{schneider}. Combining
\eqref{eq:At1} and \eqref{eq:At2} we get
$$   \int_{A_t} \sigma(x) \dd \Haus^{d-1}(x) \leq  \exp(C_\Phi(d-1) \abs{s-t}) \left(\int_{A_s} \sigma(x) \dd \Haus^{d-1}(x) + \Haus^{d-1}(\Xloc) \omega(\abs{s-t}/\eps_{nd})\right) $$
so that
\begin{align*}
  \int_{A_t} \sigma(x) \dd \Haus^{d-1}(x) - \int_{A_s} \sigma(x) \dd \Haus^{d-1}(x) &\leq  (\exp(C_\Phi(d-1) \abs{s-t})-1) \nr{\sigma}_\infty \Haus^{d-1}(\Xloc) +\\
  &\qquad \exp(C_\Phi(d-1) \abs{s-t}) \Haus^{d-1}(\Xloc) \omega(\abs{s-t}/\eps_{nd}) \\
  &\leq C_3 \cdot (\abs{s-t} + \omega(\abs{s-t}/\eps_{nd})),
\end{align*}
where the constant $C_3$ depends on $C_L$,  $\eps_{nd}$, $\eps_{tr}$,
$\nr{\sigma}_\infty$ and $\diam(\Xloc)$.  Exchanging the role of $s$
and $t$ completes the proof of \eqref{eq:A}.
\smallskip

\noindent\textbf{Proof of \eqref{eq:B}.}  By definition, for every
point $x$ in the set $B_t$, the curve $\Phi_i(x,[0,s-t])$ must cross the
boundary of $L$ at some point, so that
$$ u(x) := \min \{ v \in [0,s-t] \mid \Phi_i(x,v) \in \partial L \} $$
is well defined.  We write $P(x) := \Phi_i(x,u(x))$ for the corresponding
point on the boundary of $L$. By definition of $u(x)$, the curve
$\Phi(x, [0,u(x)])$ is included in $L$, so that by
Lemma~\ref{lemma:Phi}.(i) we have $f_i(P(x)) = t + u(x)$. This shows
\begin{align}\label{eq:boundaryinclusion}
P(B_t) \subseteq \Lambda := \partial L \cap f_{i}^{-1}([t,s]).
\end{align}
We now prove that the map $P$ satisfies a reverse-Lipschitz
inequality. Note that for any point $x$ in $B_t$,
$$ x = \Phi_i(P(x), -u(x)) = \Phi_i(P(x), t - f_i(P(x))), $$ Using the
bounds \eqref{ineq:Gronwall} and \eqref{ineq:LipTime}, we get that for
any $x,y$ in $B_t$,
\begin{align*}
  \nr{x - y} &\leq \nr{\Phi_i(P(x), t - f_i(P(x))) - \Phi_i(P(y), t - f_i(P(y)))} \\
  &\leq \nr{\Phi_i(P(x), t - f_i(P(x))) - \Phi_i(P(y), t - f_i(P(x)))} \\
&\qquad \qquad+ 
\nr{\Phi_i(P(y), t - f_i(P(x))) - \Phi_i(P(y), t - f_i(P(y)))} \\
&\leq \exp(C_\Phi T_{tr}) \nr{P(x)-P(y)} + \abs{f_i(P(x)) - f_i(P(y))}/\eps_{nd} \\
&\leq C' \nr{P(x)-P(y)}
  \end{align*}
  where $C' := \exp(C_\Phi) + C_L/\eps_{nd}$; we have used that $T_{tr}\leq 1$.
We can now bound the $(d-1)$--Hausdorff measure of $B_t$ from that of
$\Lambda$ using this Lipschitz bound and the inclusion \eqref{eq:boundaryinclusion}:
\begin{equation}
 \Haus^{d-1}(B_t) \leq \Haus^{d-1}(P^{-1}(P(B_t))) \leq C'^{d-1} \Haus^{d-1}(\Lambda). \label{eq:Bt}
\end{equation}
What remains to be done is to prove that the $(d-1)$--Hausdorff measure of
$\Lambda$ behaves like $\BigO(\abs{s-t})$, and this is where the transversality
condition will enter. 

Let us write 
\begin{align*}
 \facet_j:=
\begin{cases}
 f_j^{-1}(\lambda_j),&j\neq 0, i,\\
 \partial \Xloc\cap \partial L,&j=0.
\end{cases}
\end{align*}
Then $\partial L$ can be
partitioned (up to a $\Haus^{d-1}$--negligible set) into faces
$\partial L = \cup_{j\neq i} (\facet_j \cap L)$ and using the
coarea formula on each of the facets we get (writing $B:=f_i^{-1}([t,s])$)
\begin{align}
  \Haus^{d-1}(\Lambda) &= \sum_{j\neq i} \Haus^{d-1}(B \cap (\facet_j\cap L)) \notag \\
  &= \sum_{j\neq i} \int_{B\cap(\facet_j\cap L)}  \dd\Haus^{d-1}(x) \notag \\
  &= \sum_{j\neq i} \int_{t}^{s} \int_{S_u \cap (\facet_j \cap L)} 
\frac{1}{J_{ij}(x)} \dd\Haus^{d-2}(x) \d u,
\label{eq:Lambda}
\end{align}
where  and $J_{ij}(x)$ is the Jacobian of the restriction of the function
$f_i$ to the hypersurface $\facet_j$. More
precisely, 
$$J_{ij}(x) =
 \nr{\nabla f_i(x) - \sca{\nabla f_i(x)}{\nabla f_j(x)} \frac{\nabla f_j(x)}{\nr{\nabla f_j(x)}^2}} $$
 if $j\neq 0$, $i$, and 
\begin{align*}
 J_{i0}(x) =
 \nr{\nabla f_i(x) - \sca{\nabla f_i(x)}{v_0(x)} v_0(x)}
\end{align*}
where $v_0(x)\in \Normal_{x}\Xloc$ is a unit vector. Since $\Xloc$ is convex, for $\Haus^{d-1}$ a.e. $x\in \partial \Xloc$, the normal cone $\Normal_{x}\Xloc$ consists of only one direction, thus for such $x$ there is a unique choice of $v_0(x)$.
Let us write $v_i = \nabla f_i(x)/\nr{\nabla f_i(x)}$ and $v_j$ for either $\nabla f_j(x)/\nr{\nabla f_j(x)}$ or $v_0(x)$, we then have using \eqref{eq:T}
\begin{align}
J_{ij}(x)^2 &= \nr{\nabla f_i(x)}^2
 \nr{v_i - \sca{v_i}{v_j} v_j}^2 \notag\\
&\geq \nr{\nabla f_i(x)}^2(1- \sca{v_i}{v_j}^2)\notag\\
&\geq \eps^2_{nd}\eps^2_{tr}.\label{eq:Jij}
\end{align}
Combining \eqref{eq:Lambda} and \eqref{eq:Jij} gives us
\begin{align}
  \Haus^{d-1}(\Lambda) &\leq \frac{1}{\eps_{nd}\eps_{tr}}
  \sum_{j\neq i} \int_{t}^{s} \Haus^{d-2}(S_u \cap (\facet_j \cap L)) \d u \notag \\
  &= \frac{1}{\eps_{nd}\eps_{tr}} \int_{t}^{s} \Haus^{d-2}(S_u
  \cap \partial L) \d u. \label{eq:Lambdabis}
\end{align}
By definition, a point belongs to the intersection $S_u \cap \partial L$ if it
lies in the singularity set $\Sigma(\bm{\lambda}(u))$ where
$\bm{\lambda}(u) =
(\lambda_1,\hdots,\lambda_{i-1},u,\lambda_{i+1},\hdots,\lambda_N)$. By
Lemma~\ref{prop:HausSigma},
\begin{equation}
 \Haus^{d-2}(S_u \cap \partial L) \leq \Haus^{d-2}(\Sigma(\bm{\lambda}(u))) \leq C(d,\diam(\Xloc)) \cdot \frac{1}{\eps_{tr}}.  \label{eq:BoundSing}
\end{equation}
Combining \eqref{eq:Bt}, \eqref{eq:Lambdabis} and \eqref{eq:BoundSing}
we obtain $\Haus^d(B_t) \leq C \cdot \abs{t-s}$, which implies
\eqref{eq:B} using the boundedness of $\sigma$.
\end{proof}

\subsection{Continuity of partial derivatives}
\label{sec:continuity-PD}
We prove that the function $\Gloc$ defined in \eqref{eq:G} is continuously
differentiable by controlling the modulus of continuity of its partial
derivatives given in \eqref{eq:PD}. Again, we start with a slightly more
general proposition.

\begin{proposition} Let $\sigma$  be a 
  continuous function on $\Xloc$ with modulus of continuity
  $\omega$ and $i\in\{1,\cdots,N\}$. Consider the following function on the cube
  $Q := [-T_{tr},T_{tr}]^N$:
\begin{align*}
H(\bm{\lambda}) := \int_{K(\bm{\lambda}) \cap  f_i^{-1}(\lambda_i)}
 \sigma(x) \dd \Haus^{d-1}(x).
\end{align*} 
Then $H$ is uniformly continuous on $Q$ with modulus of continuity
$$ \omega_H(\delta) = 
C_1\cdot(\omega(C_2\delta) + \abs{\delta}),$$ where the constants
only depend on $\nr{f_i}_{\Class^{1,1}(\Xloc)}$, $\diam(\Xloc)$, $\eps_{nd}$,
$\eps_{tr}$, and $\nr{\sigma}_\infty$.
\label{prop:PDalpha}
\end{proposition}

\begin{proof} Proposition~\ref{prop:PDUC} yields that
  the function $H$ is uniformly continuous with respect to changes of
  the $i$th variable. Let us now consider variations with respect to
  the $j$th variable, with $j\neq i$ by introducing
$$ h: s\in [-T_{tr},T_{tr}] \mapsto  \int_{K(\lambda_1,\hdots,\lambda_{j-1}, s, \lambda_{j+1}, \hdots, \lambda_N) \cap f_i^{-1}(\lambda_i)}
 \sigma(x) \dd \Haus^{d-1}(x).$$
for some fixed $\bm{\lambda}\in [-T_{tr},T_{tr}]^N$.  We can
rewrite the difference between two values of $h$ using the coarea
formula. As before, we assume $s>t$ to fix the signs and introduce $L' :=
\Xloc \cap \bigcap_{k\not\in \{i,j\}} K_k(\lambda_k)$ and $S :=
f_i^{-1}(\lambda_i)$. We have
\begin{align*}
  h(s) - h(t) &= \int_{L' \cap K_j(s) \cap S} \sigma(x)
  \dd\Haus^{d-1}(x) -
  \int_{L' \cap K_j(t) \cap S} \sigma(x) \dd\Haus^{d-1}(x)  \notag \\
  &= \int_t^s \int_{L' \cap S\cap f_j^{-1}(u)} \frac{\sigma(x)}{J_{ij}(x)} \dd\Haus^{d-2}(x) \dd u,
\end{align*}
where the Jacobian factor $J_{ij}\geq \eps_{nd}\eps_{tr}$ from 
\eqref{eq:Jij}. Therefore,
\begin{equation}
h(s) \leq h(t) + \frac{\nr{\sigma}_\infty}{\eps_{nd}\eps_{tr}} \int_t^s \Haus^{d-2}(L\cap S\cap f^{-1}(u)) \dd u.
\label{eq:CoareaIJ}
\end{equation}
Just as in the proof of
Proposition~\ref{prop:PDUC}, the set $L\cap S\cap f^{-1}(u)$ is
included in the  set
$\Sigma(\lambda_1,\hdots,\lambda_{j-1},u,\lambda_j,\hdots,\lambda_N)$. Thus,
by Lemma~\ref{prop:HausSigma},
\begin{equation}
  \Haus^{d-2}(L\cap S\cap f^{-1}(u)) \leq \frac{C(d,\Xloc)}{\eps_{tr}}.
  \label{eq:HausSigmaIJ}
\end{equation}
Combining \eqref{eq:CoareaIJ} and \eqref{eq:HausSigmaIJ} we can see that the function $h$ is Lipschitz with constant 
$$C_h := C(d,\Xloc)\frac{\nr{\sigma}_\infty}
{\eps_{nd}\eps_{tr}^2}.$$
Finally,
\begin{align*}
  \abs{H(\bm{\mu}) - H(\bm{\lambda})} &\leq \sum_{j=1}^N
  \abs{H(\lambda_1,\hdots,\lambda_{k-1}, \mu_{j}, \hdots, \mu_N) - H(\lambda_1,\hdots, \lambda_{j}, \mu_{j+1}, \hdots, \mu_N)}\\
  &\leq \omega_h(\abs{\mu_i - \lambda_i}) +
  \sum_{j\neq i} C_h \abs{\mu_j - \lambda_j}\\
  &\leq \omega_h(\nr{\mu-\lambda}_{\infty}) + (N-1) C_h
  \nr{\mu-\lambda}_{\infty},
\end{align*}
where $\omega_h$ is the modulus of continuity defined in
Proposition~\ref{prop:PDUC}. This establishes the uniform continuity
of the function $H$, with the desired modulus of continuity.
\end{proof}

\subsection{Proof of Theorem~\ref{th:LocReg}}
Proposition \ref{prop:PDUC} shows that the partial derivative $\Gloc$ with
respect to the variable $\lambda_i$ exists and that its expression is given
by \eqref{eq:G:pd}. 
Applying Proposition \ref{prop:PDalpha} with $\sigma(x) =
\rholoc(x)/\nr{\nabla f_i(x)}$, we obtain $\Class^{0,\alpha}$ regularity for
each of the partial derivatives of $\Gloc$ on the cube $Q :=
[-T_{tr},T_{tr}]^N$ from the $\Class^{0,\alpha}$ regularity of $\rholoc$. Moreover, the $\Class^{0,\alpha}$ constant of each partial derivative over  $Q$
is controlled by
$$ C(\diam(\Xloc),\eps_{nd}, \eps_{tr}, \nr{\nabla f_i}_{\Class^{1,1}(X)}, \nr{\rholoc}_{\Class^{0,\alpha}(X)}).$$

\bibliographystyle{amsplain}
\bibliography{ma}

\providecommand{\bysame}{\leavevmode\hbox to3em{\hrulefill}\thinspace}
\providecommand{\MR}{\relax\ifhmode\unskip\space\fi MR }
\providecommand{\MRhref}[2]{%
  \href{http://www.ams.org/mathscinet-getitem?mr=#1}{#2}
}
\providecommand{\href}[2]{#2}
\begin{thebibliography}{10}

\bibitem{cgal}
\emph{\textsc{Cgal}, {C}omputational {G}eometry {A}lgorithms {L}ibrary},
  http://www.cgal.org.

\bibitem{geogram}
\emph{\textsc{Geograml}, a programming library of geometric algorithms},
  http://alice.loria.fr/software/geogram/doc/html/.

\bibitem{acosta2004optimal}
Gabriel Acosta and Ricardo~G Dur{\'a}n, \emph{An optimal poincar{\'e}
  inequality in $\mathrm{L}^1$ for convex domains}, Proceedings of the American
  Mathematical Society (2004), 195--202.

\bibitem{attouch2014variational}
Hedy Attouch, Giuseppe Buttazzo, and G{\'e}rard Michaille, \emph{Variational
  analysis in sobolev and bv spaces: applications to pdes and optimization},
  vol.~17, Siam, 2014.

\bibitem{aurenhammer1998minkowski}
F.~Aurenhammer, F.~Hoffmann, and B.~Aronov, \emph{{Minkowski-type theorems and
  least-squares clustering}}, Algorithmica \textbf{20} (1998), no.~1, 61--76.

\bibitem{benamou2014numerical}
Jean-David Benamou, Brittany~D Froese, and Adam~M Oberman, \emph{Numerical
  solution of the optimal transportation problem using the monge--ampere
  equation}, Journal of Computational Physics \textbf{260} (2014), 107--126.

\bibitem{caffarelli92}
Luis~A. Caffarelli, \emph{The regularity of mappings with a convex potential},
  J. Amer. Math. Soc. \textbf{5} (1992), no.~1, 99--104. \MR{1124980
  (92j:35018)}

\bibitem{caffarelli1999problem}
Luis~A Caffarelli, Sergey~A Kochengin, and Vladimir~I Oliker, \emph{Problem of
  reflector design with given far-field scattering data}, Monge Amp{\`e}re
  Equation: Applications to Geometry and Optimization: NSF-CBMS Conference on
  the Monge Amp{\`e}re Equation, Applications to Geometry and Optimization,
  July 9-13, 1997, Florida Atlantic University, vol. 226, American Mathematical
  Soc., 1999, p.~13.

\bibitem{carlier2010knothe}
Guillaume Carlier, Alfred Galichon, and Filippo Santambrogio, \emph{From
  knothe's transport to brenier's map and a continuation method for optimal
  transport}, SIAM Journal on Mathematical Analysis \textbf{41} (2010), no.~6,
  2554--2576.

\bibitem{de2014intersection}
Pedro Machado~Manhaes De~Castro, Quentin M{\'e}rigot, and Boris Thibert,
  \emph{Intersection of paraboloids and application to minkowski-type
  problems}, Proceedings of the thirtieth annual symposium on Computational
  geometry, ACM, 2014, p.~308.

\bibitem{de2012blue}
Fernando de~Goes, Katherine Breeden, Victor Ostromoukhov, and Mathieu Desbrun,
  \emph{Blue noise through optimal transport}, ACM Transactions on Graphics
  (TOG) \textbf{31} (2012), no.~6, 171.

\bibitem{de2015power}
Fernando de~Goes, Corentin Wallez, Jin Huang, Dmitry Pavlov, and Mathieu
  Desbrun, \emph{Power particles: an incompressible fluid solver based on power
  diagrams}, ACM Transactions on Graphics (TOG) \textbf{34} (2015), no.~4, 50.

\bibitem{figallikimmccann13}
Alessio Figalli, Young-Heon Kim, and Robert~J. McCann, \emph{H{\"o}lder
  continuity and injectivity of optimal maps}, Arch. Ration. Mech. Anal.
  \textbf{209} (2013), no.~3, 747--795. \MR{3067826}

\bibitem{friedland2002cheeger}
Shmuel Friedland and Reinhard Nabben, \emph{On cheeger-type inequalities for
  weighted graphs}, Journal of Graph Theory \textbf{41} (2002), no.~1, 1--17.

\bibitem{GangboMcCann1996}
Wilfrid Gangbo and Robert~J. McCann, \emph{The geometry of optimal
  transportation}, Acta Math. \textbf{177} (1996), no.~2, 113--161.
  \MR{1440931}

\bibitem{guillenkitagawa2015}
Nestor Guillen and Jun Kitagawa, \emph{On the local geometry of maps with
  c-convex potentials}, Calc. Var. Partial Differential Equations \textbf{52}
  (2015), no.~1-2, 345--387. \MR{3299185}

\bibitem{householder1975}
Alston~S. Householder, \emph{The theory of matrices in numerical analysis},
  Dover Publications, Inc., New York, 1975, Reprint of 1964 edition.
  \MR{0378371 (51 \#14539)}

\bibitem{hug1998generalized}
Daniel Hug, \emph{Generalized curvature measures and singularities of sets with
  positive reach}, Forum Mathematicum \textbf{10} (1998), no.~6, 699--728.

\bibitem{kimkitagawa14}
Young-Heon Kim and Jun Kitagawa, \emph{On the degeneracy of optimal
  transportation}, Communications in Partial Differential Equations \textbf{39}
  (2014), no.~7, 1329--1363.

\bibitem{kimmccann10}
Young-Heon Kim and Robert~J. McCann, \emph{Continuity, curvature, and the
  general covariance of optimal transportation}, J. Eur. Math. Soc. (JEMS)
  \textbf{12} (2010), no.~4, 1009--1040. \MR{2654086 (2011f:49071)}

\bibitem{kitagawa2014iterative}
Jun Kitagawa, \emph{An iterative scheme for solving the optimal transportation
  problem}, Calculus of Variations and Partial Differential Equations
  \textbf{51} (2014), no.~1-2, 243--263.

\bibitem{levy2014numerical}
Bruno L{\'e}vy, \emph{A numerical algorithm for $\mathrm{L}^2$ semi-discrete
  optimal transport in 3d}, ESAIM M2AN \textbf{49} (2015), no.~6.

\bibitem{loeper2009regularity}
Gr{\'e}goire Loeper, \emph{On the regularity of solutions of optimal
  transportation problems}, Acta mathematica \textbf{202} (2009), no.~2,
  241--283.

\bibitem{loeper2011}
\bysame, \emph{Regularity of optimal maps on the sphere: the quadratic cost and
  the reflector antenna}, Arch. Ration. Mech. Anal. \textbf{199} (2011), no.~1,
  269--289. \MR{2754343 (2011j:49081)}

\bibitem{loeper2005numerical}
Gr{\'e}goire Loeper and Francesca Rapetti, \emph{Numerical solution of the
  monge--amp{\`e}re equation by a newton's algorithm}, Comptes Rendus
  Mathematique \textbf{340} (2005), no.~4, 319--324.

\bibitem{ma2005regularity}
Xi-Nan Ma, Neil~S Trudinger, and Xu-Jia Wang, \emph{Regularity of potential
  functions of the optimal transportation problem}, Archive for rational
  mechanics and analysis \textbf{177} (2005), no.~2, 151--183.

\bibitem{merigot2011multiscale}
Quentin M{\'e}rigot, \emph{A multiscale approach to optimal transport},
  Computer Graphics Forum \textbf{30} (2011), no.~5, 1583--1592.

\bibitem{mirebeau2015discretization}
Jean-Marie Mirebeau, \emph{Discretization of the 3d monge-ampere operator,
  between wide stencils and power diagrams}, arXiv preprint arXiv:1503.00947
  (2015).

\bibitem{oliker1989numerical}
VI~Oliker and LD~Prussner, \emph{On the numerical solution of the equation and
  its discretizations, i}, Numerische Mathematik \textbf{54} (1989), no.~3,
  271--293.

\bibitem{rockafellar1970convex}
R~Tyrrell Rockafellar, \emph{Convex analysis, volume 28 of princeton
  mathematics series}, 1970.

\bibitem{saumier2015efficient}
Louis-Philippe Saumier, Martial Agueh, and Boualem Khouider, \emph{An efficient
  numerical algorithm for the l2 optimal transport problem with periodic
  densities}, IMA Journal of Applied Mathematics \textbf{80} (2015), no.~1,
  135--157.

\bibitem{schneider}
Rolf Schneider, \emph{Convex bodies: The brunn--minkowski theory}, Cambridge
  University Press, 1993.

\bibitem{trudingerwang2009}
Neil~S. Trudinger and Xu-Jia Wang, \emph{On the second boundary value problem
  for {M}onge-{A}mp\`ere type equations and optimal transportation}, Ann. Sc.
  Norm. Super. Pisa Cl. Sci. (5) \textbf{8} (2009), no.~1, 143--174.
  \MR{2512204 (2011b:49121)}

\bibitem{villani2009optimal}
C.~Villani, \emph{{Optimal transport: old and new}}, Springer Verlag, 2009.

\end{thebibliography}
\end{document}